\documentclass[11pt,reqno]{amsart}

\usepackage{amsmath, amsthm, amsopn, amssymb}
\usepackage[nobysame,msc-links]{amsrefs}
\usepackage{enumerate, tikz, etoolbox, intcalc, geometry, caption, subcaption, afterpage}
\usepackage[english]{babel}

\newtheorem{thm}{Theorem}[section]
\newtheorem{lemma}[thm]{Lemma}
\newtheorem{prop}[thm]{Proposition}

\theoremstyle{definition}
\newtheorem{remark}{Remark}

\newcommand{\cC}{\mathcal{C}}

\newcommand{\cF}{\mathcal{F}}
\newcommand{\cE}{\mathcal{E}}

\newcommand{\cS}{\mathcal{S}}

\newcommand{\N}{\mathbb{N}}

\newcommand{\Z}{\mathbb{Z}}

\newcommand{\V}{\mathbb{V}}
\newcommand{\E}{\mathbb{E}}

\renewcommand{\Pr}{\mathbb{P}}
\newcommand{\1}{\mathbf{1}}
\def\eqd{\,{\buildrel d \over =}\,}
\newcommand{\zero}{\mathbf{0}}

\DeclareMathOperator\dist{dist}

\newcommand{\iid}{i.i.d.}

\title{Finitely dependent processes are finitary}
\date{\today}

\author{Yinon Spinka}
\address{University of British Columbia.
    Department of Mathematics.
 	Vancouver, BC V6T 1Z2, Canada.}
    \email{yinon@math.ubc.ca}

\begin{document}

\begin{abstract}
We show that any finitely dependent invariant process on a transitive amenable graph is a finitary factor of an \iid\ process.
With an additional assumption on the geometry of the graph, namely that no two balls with different centers are identical, we further show that the \iid\ process may be taken to have entropy arbitrarily close to that of the finitely dependent process.
As an application, we give an affirmative answer to a question of Holroyd~\cite{holroyd2017one}.
\end{abstract}

\maketitle

\section{Introduction}\label{sec:introduction}

Consider a random process $X=(X_v)_{v\in\V}$ living on the vertex set $\V$ of an infinite graph $G$. The process~$X$ is said to be \emph{finitely dependent} if its restrictions to sets which are sufficiently separated (at least some fixed distance apart) are independent. A trivial example of a finitely dependent process is a process $Y=(Y_v)_{v \in \V}$ in which all random variables are independent. A natural question is then how close is a finitely dependent process to such an independent process?
Before addressing this question, we first observe that ``local functions'' of an independent process $Y$ are always finitely dependent. That is, if $X$ is obtained from $Y$ by computing each $X_v$ as a function only of the random variables $Y_u$ for which $u$ is at a uniformly bounded distance from $v$, then $X$ is finitely dependent.

Suppose now that $G$ is transitive and henceforth restrict attention to processes $X$ which are invariant under all automorphisms of $G$ (or under a transitive subgroup of automorphisms). In particular, the independent process $Y$ considered above must now be an \iid\ process -- that is, in addition to being independent, the $\{Y_v\}_v$ are also identically distributed.
If $X$ is obtained from~$Y$ by applying the same local function at each vertex $v$ (i.e., the function applied at $u$ is the composition of the function applied at $v$ with any automorphism taking $u$ to $v$), then $X$ is said to be a \emph{block factor} of $Y$.
Thus, block factors of \iid\ processes provide a recipe for constructing invariant finitely dependent processes.

It was a long-standing open problem~\cite{ibragimov1965independent,ibragimov1975independent} to determine whether block factors of \iid\ processes are the only (invariant) finitely dependent processes on $\Z$, until finally an example was given by Burton--Goulet--Meester~\cite{burton19931} of a 1-dependent process which is not a block factor of any \iid\ process.
Recently, Holroyd and Liggett~\cite{holroyd2016finitely} showed that proper colorings distinguish between block factors of \iid\ processes and finitely dependent processes -- no proper coloring of $\Z$ is a block factor of an \iid\ process, but finitely dependent proper colorings exist.

Thus, it is not true that every finitely dependent process is a block factor of an \iid\ process. In other words, given a finitely dependent process $X$, one cannot in general hope to find an \iid\ process $Y$ and an invariant rule for computing $X$ from $Y$, which allows to determine the value of $X_v$ by looking at $Y$ on a \emph{fixed-size} window around $v$.
The goal of this paper is to show the ``next best thing'' -- namely, that it is possible to determine $X_v$ by looking at $Y$ on a \emph{variable-sized} window around $v$, where the size of the window, though always finite, may vary according to the input $Y$.

We say that $X$ is a \emph{finitary factor} of $Y$ if there is an invariant rule which allows to compute the value of $X$ at any vertex $v$ by only looking at variables $Y_u$ for which $u$ is within a certain finite, but random, distance from $v$ (formal definitions are given below). Thus, a block factor is a finitary factor in which the required distance is not only finite, but is determistically bounded by some constant.
The main contribution of this paper is to prove that every finitely dependent process is a finitary factor of an \iid\ process.
This result holds on any amenable graph $G$. When it is further assumed that no two balls in $G$ with different centers are identical, it is also possible to control the entropy of the \iid\ process involved, and the result becomes that every finitely dependent process is a finitary factor of an \iid\ process with only slightly larger entropy.

\subsection{Definitions and main result}
Let $\V$ be a countable set, let $G$ be a graph on vertex set $\V$ and let $\Gamma$ be a group acting on $\V$. A \emph{random field (or random process)} on $G$ is a collection of random variables $X=(X_v)_{v \in \V}$ indexed by the vertices of $G$ and defined on a common probability space.
We say that $X$ is \emph{$\Gamma$-invariant} if its distribution is not affected by the action of $\Gamma$, i.e., if $(X_{\gamma v})_{v \in \V}$ has the same distribution as $X$ for any $\gamma \in \Gamma$. We say that $X$ is \emph{$k$-dependent} if $(X_u)_{u \in U}$ and $(X_v)_{v \in V}$ are independent for any two sets $U,V \subset \V$ such that $\dist(u,v)>k$ for all $u \in U$ and $v \in V$. We say that $X$ is \emph{finitely dependent} if it is $k$-dependent for some finite $k$.

Suppose now that $G$ is a transitive locally finite graph and that $\Gamma$ is a subgroup of the automorphism group of $G$.
Let $S$ and $T$ be two measurable spaces, and let $X=(X_v)_{v \in \V}$ and $Y=(Y_v)_{v \in \V}$ be $S$-valued and $T$-valued $\Gamma$-invariant random fields.
A \emph{coding} from $Y$ to $X$ is a measurable function $\varphi \colon T^\V \to S^\V$ that is \emph{$\Gamma$-equivariant}, i.e., commutes with the action of every element of $\Gamma$, and satisfies that $\varphi(Y)$ and $X$ are identical in distribution. Such a coding is also called a \emph{factor map} from $Y$ to~$X$, and when such a coding exists, we say that $X$ is a \emph{$\Gamma$-factor} of~$Y$.

Suppose now that $S$ and $T$ are countable.
Let $\zero \in \V$ be a distinguished vertex.
The \emph{coding radius} of $\varphi$ at a point $y \in T^\V$, denoted by $R(y)$, is the minimal integer $r \ge 0$ such that $\varphi(y')_\zero=\varphi(y)_\zero$ for all $y' \in T^\V$ which coincide with $y$ on the ball of radius $r$ around $\zero$ in the graph-distance, i.e., $y'_v=y_v$ for all $v \in \V$ such that $\dist(v,\zero) \le r$. It may happen that no such~$r$ exists, in which case, $R(y)=\infty$.
Thus, associated to a coding is a random variable $R=R(Y)$ which describes the coding radius.
While $S$ will always be at most countable, we will allow $T$ to be a larger space, in which case the coding radius may be similarly defined\footnote{We will only be concerned with spaces $T$ which are finite, countable or of the form $T_1 \times T_2 \times \cdots$ for finite sets $(T_i)$. In the latter case, the coding radius is the smallest $r$ for which there exists $n$ such that $\varphi(y')_\zero=\varphi(y)_\zero$ for all $y'$ having the property that $y'_{v,i}=y_{v,i}$ for all $(v,i)$ such that $\dist(v,\zero) \le r$ and $1 \le i \le n$.}.
A coding is called \emph{finitary} if $R$ is almost surely finite. When there exists a finitary coding from $Y$ to $X$, we say that $X$ is a \emph{finitary $\Gamma$-factor} of $Y$.

A graph is said to be \emph{amenable} if $\inf |\partial V|/|V| = 0$, where the infimum is over all finite non-empty subsets $V$ of $\V$, and where $\partial V$ denotes the \emph{edge-boundary} of $V$.

\begin{thm}\label{thm:main-no-entropy}
	Let $G$ be a transitive amenable graph and let $\Gamma$ be a transitive group of automorphisms of~$G$. Then any finitely dependent $\Gamma$-invariant random field on $G$ is a finitary $\Gamma$-factor of an \iid\ process.
\end{thm}

With a minor additional constraint on the geometry of the graph $G$, we can further control the entropy of the \iid\ process used in the coding (see Section~\ref{sec:entropy} for the definition of entropy).
The condition we require is that
\begin{equation}\label{eq:sphere-condition}
\Lambda_r(u) \neq \Lambda_r(v) \qquad\text{for any distinct }u,v \in \V \text{ and } r\geq 0,
\end{equation}
where $\Lambda_r(u)$ is the ball of radius $r$ around $u$.

\begin{thm}\label{thm:main}
	Let $G$ be a transitive amenable graph satisfying~\eqref{eq:sphere-condition}, let $\Gamma$ be a transitive group of automorphisms of $G$, and let $X$ be a finite-valued finitely dependent $\Gamma$-invariant random field on $G$. Then for any $\epsilon>0$ there exists an \iid\ process~$Y$ with entropy $h(Y)<h(X)+\epsilon$ such that $X$ is a finitary $\Gamma$-factor of~$Y$.
\end{thm}

Let us make some remarks about how the two theorems compare to one another. In Theorem~\ref{thm:main}, $S$ is finite (note the assumption that $X$ is finite-valued) and, in particular, $X$ has finite entropy, while in Theorem~\ref{thm:main-no-entropy}, $S$ may be countable and $X$ may have infinite entropy. In Theorem~\ref{thm:main}, $Y$ has finite entropy so that $T$ is countable, whereas Theorem~\ref{thm:main-no-entropy} may require a larger space $T$ for the conclusion to hold. In fact, in the absence of condition~\eqref{eq:sphere-condition}, even when $S$ is finite, it might not be possible to have $T$ countable (see Remark~\ref{rem:spere-condition} in Section~\ref{sec:conclusion}).

Let us also remark that, while the theorems do not assume connectivity of the graph, there is no loss of generality in assuming this, since transitivity implies that the connected components are isomorphic and finite dependence implies that the random field is independent on different components. Thus, the same coding can be used for all components.

\subsection{Discussion}
\label{sec:discussion}

Finite dependence and finitary factors have applications in computer science. For example, if the graph $G$ represents machines in a network and the random field $X$ represents a common plan in which each machine $v$ is assigned a specific role $X_v$, then finite dependence provides certain security benefits in the face of an attacker (if someone gains access to some machines, they learn nothing about the roles of far away machines, thereby confining the security breach), and being a finitary factor of an \iid\ process provides reliability (e.g., no single point of failure) as it means that the machines can determine their own roles in a distributed manner by following a common protocol, while using local randomness and communicating with finitely many other machines. See e.g.~\cite{linial1987distributive,naor1991lower} for more information.

The finitary coding properties of finitely dependent processes on $G=\Z$, and in some cases on $G=\Z^d$ with $d \ge 2$, have been studied in various contexts. We give a brief account of these works. We are unaware of any works regarding finitary factors for finitely dependent processes on other graphs.

A result by Smorodinsky~\cite{smorodinsky1992finitary} shows that every stationary (i.e., translation-invariant) finitely dependent process on $\Z$ is finitarily isomorphic (a stronger notion than being a finitary factor) to an \iid\ process. This result is not for the full autormorphism group of the graph (which includes reflections), but rather only for the group of translations. In this respect, Theorem~\ref{thm:main} strengthens this result (if one is content with a finitary factor, rather than a finitary isomorphism), as it provides a finitary factor which is also reflection invariant whenever the finitely dependent process is such. The proof in~\cite{smorodinsky1992finitary} is based on the so-called marker-filler method of Keane and Smorodinsky~\cite{keane1977class,keane1979bernoulli}. Unfortunately, only a brief sketch of the proof is provided in~\cite{smorodinsky1992finitary} and the details seem to be missing (after some initial steps, Smorodinsky says that the rest of the proof proceeds along the same lines as in~\cite{keane1979bernoulli} with some necessary modifications). Our proof is based on a different approach; see Section~\ref{sec:outline} for an outline.

The question of whether there exists a stationary finitely dependent process which is not a block factor of any \iid\ process was raised by Ibragimov and Linnik~\cite{ibragimov1965independent,ibragimov1975independent} in 1965.
Some progress on this question was made~\cite{aaronson1989algebraic,aaronson1992structure} until it was finally resolved in 1993 by Burton--Goulet--Meester~\cite{burton19931} who gave the first example of a stationary finitely dependent process which is not a block factor of an \iid\ process. In fact, they showed such an example in which the finitely dependent process is a 1-dependent hidden-Markov process with finite energy. Some history about finitely dependent processes that cannot be written as block factors is given in~\cite{holroyd2016finitely}.

Holroyd and Liggett~\cite{holroyd2016finitely} constructed a stationary 1-dependent 4-coloring and a stationary 2-dependent 3-coloring of $\Z$, neither of which is a block factor of an \iid\ process (indeed, no coloring is such~\cite{holroyd2016finitely}).
Holroyd~\cite{holroyd2017one} subsequently showed that the 1-dependent 4-coloring is a finitary factor of an \iid\ process. Regarding the analogous statement for the 2-dependent 3-coloring, Holroyd writes in~\cite{holroyd2017one} that ``one may attempt to apply our method to the 2-dependent 3-coloring, but we will see that it meets a fundamental obstacle in this case''. Our result shows that either coloring is a finitary factor of an \iid\ process (with slightly larger entropy), answering affirmatively question~(iii) in~\cite[Open~problems]{holroyd2017one}. In fact, the two colorings are also reflection invariant, and hence the finitary factors may also be taken to commute with reflections. Related aspects of these two colorings were studied in~\cite{holroyd2018finitely}.

In a subsequent paper~\cite{holroyd2015symmetric}, Holroyd and Liggett constructed, for any $q \ge 4$, a 1-dependent $q$-coloring of $\Z$ which is invariant under translations and reflections and is also symmetric under permutations of the colors. It was shown in~\cite{holroyd2017mallows} that each of these colorings is a finitary factor of an \iid\ process (with exponential tail on the coding radius).
Our result shows that each of these colorings is a finitary factor of an \iid\ process, where the factor map commutes with all automorphisms of $\Z$ and the \iid\ process has entropy only slightly larger than the coloring.

In~\cite{holroyd2016finitely}, Holroyd and Liggett also constructed stationary finitely dependent colorings of $\Z^d$ with $d \ge 2$. More specifically, they constructed a stationary 1-dependent $4^d$-coloring of $\Z^d$ and a stationary finitely dependent 4-coloring of $\Z^d$. However, unlike the above one-dimensional colorings, these colorings are only translation-invariant and not automorphism-invariant. In fact, it is still unknown whether there exists a finitely dependent coloring of $\Z^d$ ($d \ge 2$) which is invariant under all automorphisms of $\Z^d$.

In the same paper~\cite{holroyd2016finitely}, Holroyd and Liggett also investigated the existence of stationary finitely dependent processes on $\Z$ which are supported on a given shift of finite type. They showed that for any reasonably non-degenerate (namely, nonlattice) shift of finite type $\cS$ on $\Z$, there exists a stationary finitely dependent process which almost surely belongs to $\cS$. It was later shown~\cite{holroyd2017mallows} that there exists such a process which is also a finitary factor of an \iid\ process (with exponential tail on the coding radius).

A block factor is precisely a finitary factor with bounded coding radius. Given a finitary factor which is not a block factor, it is natural to wonder about the typical value of the coding radius. As we have mentioned, the 1-dependent 4-coloring of $\Z$ from~\cite{holroyd2016finitely}, which is not a block factor of any \iid\ process, was shown in~\cite{holroyd2017one} to be a finitary factor of an \iid\ process. This finitary factor was shown to have (at least) power-law tail on the coding radius, thus yielding a perhaps infinite expected coding radius. Holroyd--Hutchcroft--Levy~\cite{holroyd2017mallows} showed that there exist finitely dependent colorings of $\Z$ which are finitary factors of \iid\ processes with exponential tail on the coding radius. Indeed, they showed that such a $k$-dependent $q$-coloring exists when $(k,q)$ is either $(1,5)$, $(2,4)$ or $(3,3)$. On the other hand, it is believed~\cite{holroyd2017one,holroyd2017mallows} that when $(k,q)$ is $(1,4)$ or $(2,3)$, no $k$-dependent $q$-coloring is a finitary factor of an \iid\ process with finite expected coding radius.
We mention that optimal tails for the coding radius of colorings of $\Z^d$ (which are not necessarily finitely dependent) and shifts of finite type on $\Z$ have been studied in~\cite{holroyd2017finitary}.

Our main theorem gives no information about the coding radius beyond its almost-sure finiteness. Indeed, in light of the above discussion, it would seem that for an arbitrary finitely dependent process on~$\Z$, there is not much hope to obtain a finitary factor with finite expected coding radius. Still, some information about the coding radius may be extracted from the proof given here (see~Remark~\ref{rem:coding-radius}). For example, for 1-dependent processes on $\Z$, the finitary factor provided by Theorem~\ref{thm:main-no-entropy} has a coding radius $R$ satisfying that $\Pr(R > r) \le 8/r$ for all $r$. In the particular case of the 1-dependent 4-coloring of~\cite{holroyd2016finitely}, this improves the power in the power-law bound shown in~\cite{holroyd2017one} (to an optimal power if the prediction above is indeed correct).

To the best of our knowledge, beyond Smorodinsky's result on $\Z$, there do not exist any general results on the finitary coding properties of finitely dependent processes. In particular, Theorem~\ref{thm:main-no-entropy} and Theorem~\ref{thm:main} are new for any amenable graph $G$ other than $\Z$, and also for $G=\Z$ in the case when $\Gamma$ is the full automorphism group of $\Z$.
Finally, we mention that the situation for non-amenable graphs is still poorly understood -- for example, on a regular tree (of degree at least three), it is not even known whether every automorphism-invariant finitely dependent process is a (non-finitary) factor of an \iid\ process~\cite{lyons2017factors}.

\subsection{Acknowledgments}
I would like to thank Omer Angel, Nishant Chandgotia, Tom Meyerovitch and Mathav Murugan for useful discussions.
I am especially grateful to Nishant Chandgotia for suggesting to extend the result from $\Z^d$ to transitive amenable graphs, and to Omer Angel for jointly proving Lemma~\ref{lem:invariant-Folner-sequence} with me. I would also like to thank the referees for useful comments.

\subsection{Notation}
Throughout the paper, $G$ is always assumed to be an infinite, transitive, locally finite, connected graph on a countable vertex set $\V$, and $\Gamma$ is a subgroup of the automorphism group of $G$ that acts transitively on $\V$. The full automorphism group of $G$ is denoted by $\text{Aut}(G)$.
The graph distance in $G$ is denoted by $\dist(\cdot,\cdot)$.
For sets $U,V \subset \V$, we write $\dist(U,V) := \min_{u \in U,v \in V} \dist(u,v)$ and $\dist(u,V) := \dist(\{u\},V)$. For $r \ge 0$, denote $V^{+r} := \{ u \in \V : \dist(u,V) \le r\}$ and $V^{-r} := \{ u \in \V : \dist(u,V^c) > r \}$.
The ball of radius $r$ around $v$ is denoted by $\Lambda_r(v) := \{v\}^{+r}$. The \emph{neighborhood} of~$V$ is $N(V) := V^{+1} \setminus V$ and the \emph{edge-boundary} of $V$ is $\partial V := \{ \{u,v\} \in E(G) : v \in V, u \notin V \}$.

All logarithms are taken to be in base 2 and we use the convention that $0\log 0$ is~0.

\section{Outline of proof}
\label{sec:outline}

Our goal is to express $X$, a finitely dependent invariant process, as a finitary factor of an \iid\ process $Y$.
The construction of the finitary coding involves the use of three sources of randomness: a random number of random bits located at each vertex, a so-called cell process, and a random total order on $\V$. The first of these three will simply be given by an \iid\ process, denoted $Y^{\text{bits}}$, while the latter two will be obtained as finitary factors of different \iid\ processes, denoted $Y^{\text{cell}}$ and $Y^{\text{ord}}$. In turn, $Y$ will be a triplet $Y=(Y^{\text{bits}},Y^{\text{cell}},Y^{\text{ord}})$ consisting of three mutually independent \iid\ processes.

To illuminate the main ideas behind our construction, we provide a sketch of the proof below, explaining separately three ingredients:
\begin{enumerate}
 \item \textbf{Constructing a finitary coding:} The basic and most essential part of the construction is how to obtain $X$ as a finitary factor of $Y$ when $Y$ is allowed to have infinite entropy (this is the setting of Theorem~\ref{thm:main-no-entropy}). In this case, $Y^{\text{bits}}_v$ and $Y^{\text{ord}}_v$ may be taken to be uniform random variables in $[0,1]$, and the total order may be taken to be the one induced by the usual order on $Y^{\text{ord}}_v$.
 \item \textbf{Controlling the entropy:} The second part is how to control the entropy of the $Y^{\text{bits}}$ process, requiring only slightly more entropy than that of $X$. To postpone dealing with the issue of controlling the entropy of $Y^{\text{ord}}$, we shall assume in this part of the proof outline that the vertices of $G$ can be deterministically ordered in a $\Gamma$-invariant manner so that $Y^{\text{ord}}$ may be disregarded entirely (e.g., it can be taken to be a constant process). For example, the lexicographical order is such an ordering when $G$ is the graph $\Z^d$ and $\Gamma$ is the group of translations.
 \item \textbf{Constructing a random order:} The third part is how to allow for graphs $G$ and groups~$\Gamma$ which do not admit such a deterministic order. This is of course the case for general graphs, but it may also be the case for simpler graphs, such as $\Z$ or $\Z^d$, when $\Gamma$ is the full automorphism group of $G$. In these cases, we are led to consider random orders with suitable properties.
\end{enumerate}

Already the first part above relies on the aforementioned cell process. Before introducing this process, it is convenient to observe that it suffices to prove Theorem~\ref{thm:main-no-entropy} and Theorem~\ref{thm:main} for 1-dependent random fields. To see this, let $G^{\otimes k}$ denote the graph on vertex set $\V$ in which two vertices are adjacent if their distance in $G$ is at most $k$. It is immediate from the definitions that $X$ is $k$-dependent as a random field on $G$ if and only if it is 1-dependent as a random field on~$G^{\otimes k}$.
Since any automorphism of $G$ is also an automorphism of~$G^{\otimes k}$, and since $G^{\otimes k}$ is amenable and satisfies~\eqref{eq:sphere-condition} whenever $G$ is such, we see that we may indeed assume that $X$ is 1-dependent. This assumption, though not at all essential, is convenient as it obviates the need to work with a different connectivity than the usual connectivity in $G$.

A \emph{cell process} is a random sequence $A=(A_1,A_2,\dots)$ of subsets of $\V$ satisfying the following properties almost surely:
\begin{itemize}
	\item $A_1 \subset A_2 \subset A_3 \subset \cdots$.
	\item $A_1 \cup A_2 \cup \dots = \V$.
	\item For each $n \ge 1$, all connected components of $A_n$ are finite.
\end{itemize}
We will obtain a cell process $A$ as a finitary factor of an \iid\ process $Y^{\text{cell}}$ with arbitrarily small entropy. We do not explain here how this is done and refer the reader to Section~\ref{sec:cell-process} for more details and to Figure~\ref{fig:cell-process} for an illustration of the construction.

\smallskip
\noindent
{\bf (1) Constructing a finitary coding:}
We construct a realization of $X$ as a finitary factor of $Y$ in infinitely many steps with the idea that at the end of step $n \in \{1,2,\dots\}$, we will have defined $X$ on the region $A_n$.
In the first step, we sample $X$ on the set $A_1$ -- this is particularly simple as the cells in $A_1$ are at pairwise distance at least 2, and so, due to the 1-dependence assumption on $X$, each cell in $A_1$ can be sampled independently.
Next, at each step $n \in \{2,3,\dots\}$, we sample $X$ on the region $A_n \setminus A_{n-1}$, conditioned on the value of $X$ on $A_{n-1}$, which has already been sampled in the previous steps -- the key observation here is that, due again to the 1-dependence assumption on~$X$, the values of $X$ on different cells in $A_n$ are conditionally independent -- indeed, if $\{V_i\}_i$ are at pairwise distance at least 2 from one another, then for any sets $U_i \subset V_i$, given $\{X_{U_i}\}_i$, one has that $\{X_{V_i}\}_i$ are mutually conditionally independent.
Since all the cells of every $A_n$ are finite, the above steps can be carried out in a finitary manner -- that is, the conditional distribution of $X$ on a given cell depends only on the previously sampled values within that cell.
Since $A_n$ increases to~$\V$, the value at every given vertex will eventually be sampled, thus producing a realization of $X$ from $Y$ in a finitary and $\Gamma$-equivariant manner.

Let us be slightly more specific about the way in which we ``sample $X$ on a cell''.
In each cell in $A_1$, we distinguish a vertex by choosing the smallest element in the cell according to the order given by $Y^{\text{ord}}$.
We call these distinguished vertices \emph{level 1 agents}.
Similarly, for each $n \ge 2$ and each cell in $A_n$ that is not contained in $A_{n-1}$, we select a \emph{level $n$ agent} in the cell by choosing the smallest element in the cell which is not in $A_{n-1}$.
Note that the level $n$ agents are obtained as a finitary factor of $(Y^{\text{cell}},Y^{\text{ord}})$.
With the notion of agents, we may now say more precisely that, in step $n$ above, if $\cC$ is a cell of $A_n$ that is not contained in $A_{n-1}$, then we sample $X$ on the region $\cC \setminus A_{n-1}$ (conditionally on the previously sampled values of $X$ on $\cC \cap A_{n-1}$) by accessing a sample of the desired distribution from the random variable $Y^{\text{bits}}_u$, where $u$ is the unique level $n$ agent in $\cC$, using also the order induced by $Y^{\text{ord}}$ on $\cC \setminus A_{n-1}$ to break any symmetries which may be present in the graph structure of this region (for example, if $G=\Z$ and $\Gamma$ includes reflections, then when $\cC \setminus A_{n-1}$ is a symmetric interval around $u$, the `left' and `right' sides of $u$ cannot be differentiated in a $\Gamma$-equivariant way without some additional information; ordering all elements in the set is a simple way to get rid of such problems). In this interpretation, we regard $Y^{\text{bits}}_u$ as consisting of independent samples of $\Pr(X_U \in \cdot \mid X_V=\tau)$ for all finite $U,V \subset \Z^d$ and $\tau \in S^V$, most of which are never used in practice.

\smallskip
\noindent
{\bf (2) Controlling the entropy:}
It is clear from the last observation above that there is plenty of waste in the above construction (in terms of the process $Y^{\text{bits}}$).
The problem is that we do not know ahead of time which samples of which distributions we will need access to. The basic solution to this is to place an infinite sequence of random bits at each site, i.e., $Y^{\text{bits}}_v \in \{0,1\}^\N$, from which we may easily construct samples of any desired distributions (hence the name of the process $Y^{\text{bits}}$). Of course, this idea alone still does not provide any control on the entropy of~$Y^{\text{bits}}$. For this, we must place a finite (perhaps random) number of random bits at each site, and somehow still be sure that we are able to construct the required samples. By a random number of random bits, we mean a random variable $W$ taking values in $\{0,1\}^*$, the set of finite words over $\{0,1\}$, and having the property that, conditioned on the length $|W|$ of the word, $W$ is uniformly distributed on $\{0,1\}^{|W|}$.

Suppose now that there exists a deterministic total order $\le$ on $\V$ that is $\Gamma$-invariant in the sense that $u \le v$ implies that $\gamma u \le \gamma v$ for any $u,v \in \V$ and $\gamma \in \Gamma$. For instance, the lexicographical order is such an order when $G=\Z^d$ and $\Gamma$ is the translation group (but there is clearly no such order when $\Gamma$ is the full automorphism group of $\Z^d$).
For the purpose of this part of the proof outline, it is convenient to further suppose that every $v \in \V$ has a $\le$-successor, which we denote by $v+1$, as is the case for the lexicographical order on $\Z^d$ (in which case $v+1$ is simply $v+(1,0,\dots,0)$). The existence of such a deterministic order renders $Y^{\text{ord}}$ unneeded, allowing us to focus now only on the task of controlling the entropy of $Y^{\text{bits}}$.

The idea is to associate to each possible distribution we might require, a ``simulation'' which outputs a sample of the distribution in question from an input of unbiased random bits. The simulation is fed independent unbiased bits one-by-one, until at some point (a stopping time) it halts and outputs the sample. Such simulations may be done efficiently: the expected number of input bits read by the simulation is bounded by the entropy of the target distribution, up to an additive universal constant. 

We shall use such simulations whenever we ``sample $X$ on a cell''.
If the cell $\cC$ is large, then the entropy of $X$ on $\cC$ is also large, and thus the above additive error is negligible. When the boundary of the cell is small in comparison to the size of the cell, the average entropy of $X$ on $\cC$ per site will also be close to $h(X)$, the entropy of $X$ itself. Thus, it will be important that the cells in $A_1$ are typically large with small boundary.

This already shows that, in some sense, the average number of random bits needed to generate the samples required throughout the construction is very close to $h(X)$. However, we must place a finite number of bits at each site (more precisely, we need that $h(Y^{\text{bits}})<h(X)+\epsilon$), and even if we have more than $h(X)$ such bits at every site, it still may happen at some point during the construction that a simulation carried out by an agent $u$ requires access to many more input bits than are available in $Y^{\text{bits}}_u$. To solve this, we must allow to ``transfer'' bits from one location to another. This aspect of our construction is inspired by the algorithms in~\cite{van1999existence,harvey2006universal,spinka2018finitaryising}.
The idea is that whenever an agent $u$ requires access to an additional bit (beyond those available in $Y^{\text{bits}}_u$), it may look for an ``unused'' bit at $u+1$ (the $\le$-successor of $u$). If there are no available unused bits at $u+1$ at that time, it may then proceed to look at $u+2$, and so on.

One consequence of the above description is that the steps of the construction cannot be directly related to the levels of the cell process. That is, it will no longer be the case that after step $n$ of the construction, we will have defined $X$ on the region $A_n$. Instead, at any step of the construction, different regions of $G$ will be at different levels of the cell process. We will continue to use $n$ to denote the levels of the cell process, and will use $t$ to denote the step of the construction (which we henceforth also refer to as time).

The way this is done is as follows. Initially, at time $t=0$, all level 1 agents are deemed active. An active level 1 agent attempts to collect unused bits until its associated simulation halts, at which point in time the agent becomes inactive and is said to have completed level 1. Once all level 1 agents contained in some level 2 cell $\cC$ have completed, the level 2 agent associated to $\cC$ becomes active. An active level 2 agent proceeds in the same manner as an active level 1 agent, attempting to read bits in order to complete its associated simulation. In general, a level $n$ agent becomes active once all level $n-1$ agents contained in its associated cell have completed.

In our actual construction, it is more convenient to employ the following policy which makes the details simpler to write down: at time $t$, an agent $u$ may read at most one bit, and this bit may only be read from site $u+t$. This has the advantage that it ensures that no two agents ever try to read bits from the same location simultaneously.

\smallskip
\noindent
{\bf (3) Constructing a random order:}
For a general graph $G$ and group $\Gamma$, there need not be a deterministic total order of $\V$ that is $\Gamma$-invariant. Instead, we construct a random total order $\le$ on $\V$ whose distribution is $\Gamma$-invariant. Moreover, we construct $\le$ as a finitary factor of an \iid\ process $Y^{\text{ord}}$ with arbitrarily small entropy. Here finitary means that the order induced on any finite set of vertices is determined by a finite (random) subset of $\{Y^{\text{ord}}_v\}_{v \in \V}$.

In addition, the constructed order $\le$ will have the property that its order type is almost surely the same as that of $\Z$. That is, almost surely, every element $v$ has a successor $v+1$ and a predecessor $v-1$, and $\{v+n\}_{n \in \Z} = \V$. Though it does not follow from the above definition of finitary, it will turn out to be the case that determining whether some vertex is the successor of some other vertex is also a finitary property (i.e., it is almost surely determined by a finite subset of $\{Y^{\text{ord}}_v\}_{v \in \V}$).
Once such an order is at hand, the proof continues as outlined above.

\medbreak
\noindent
{\bf Organization.}
In Section~\ref{sec:preliminaries}, we introduce some preliminaries.
In Section~\ref{sec:cell-process}, we prove the existence of a finitary cell process.
In Section~\ref{sec:random-orders}, we prove the existence of a finitary random total order having the order type of $\Z$.
In Section~\ref{sec:coding}, we give the construction of the finitary coding for the finitely dependent process $X$.
Finally, we end in Section~\ref{sec:conclusion} with some remarks and open problems.

\section{Preliminaries}
\label{sec:preliminaries}

Recall that $G$ is always assumed to be an infinite, transitive, locally finite, connected graph on a countable vertex set $\V$, and that $\Gamma$ is assumed to be a subgroup of the automorphism group of $G$ that acts transitively on $\V$.

\subsection{Entropy}
\label{sec:entropy}

The \emph{Shannon entropy} of a discrete random variable $Z$ is
\[ H(Z) := - \sum_z \Pr(Z=z) \log \Pr(Z=z) ,\]
where the sum is taken over $z$ in the support of $Z$, or alternatively, we interpret $0\log 0$ to be 0.
The \emph{measure-theoretic entropy} (or Kolmogorov--Sinai entropy) of a $\Gamma$-invariant random field $X$ on an amenable graph $G$ is
\[ h(X) := \inf_{\substack{V \subset \V\text{ finite}\\\text{and non-empty}}} \frac{H(X_V)}{|V|} .\]
A \emph{F{\o}lner sequence} in $G$ is a sequence $(F_n)_{n=1}^\infty$ of non-empty finite subsets of $\V$ such that
\[ \lim_{n \to \infty} \frac{|\partial F_n|}{|F_n|}=0 .\]
It is well-known that the entropy of $X$ may be computed along any F{\o}lner sequence:
\[ h(X) = \lim_{n \to \infty} \frac{H(X_{F_n})}{|F_n|} \qquad\text{for any F{\o}lner sequence }(F_n)_{n=1}^\infty\text{ in }G .\]
It follows that for any $\epsilon>0$ there exists $\delta>0$ such that
\begin{equation}\label{eq:entropy-via-boundary}
\frac{H(X_F)}{|F|} \le h(X)+\epsilon\qquad \text{whenever }F \subset \V\text{ is non-empty and finite and }|\partial F| \le \delta |F| .
\end{equation}
Of course, since entropy is maximized by the uniform distribution, we also have that
\begin{equation}\label{eq:entropy-trivial-bound}
\frac{H(X_F \mid \cE)}{|F|} \le \log|S| \qquad\substack{\text{whenever }F \subset \V\text{ is non-empty and finite}\\\text{and $\cE$ is an event with positive probability}} ,
\end{equation}
where $S$ is the finite set in which $X$ takes values.
We note that if $Y$ is an \iid\ process, then its entropy $h(Y)$ is equal to the entropy of its single-site distribution $H(Y_\zero)$.

\subsection{The mass-transport principle}
\label{sec:mass-transport}

For $u,v \in \V$, denote
\[ \Gamma_{u,v} := \{ \gamma \in \Gamma : \gamma u = v \} .\]
Note that $\Gamma_{u,u}$ is the \emph{stabilizer} of~$u$.
We say that $\Gamma$ is \emph{unimodular} if
\[ |\Gamma_{u,u}v|=|\Gamma_{v,v}u| \qquad\text{for all }u,v \in \V .\]
It is well-known (see, e.g., \cite[Chapter~8]{lyons2017probability}) that $\Gamma$ is unimodular if and only if the following \emph{mass-transport principle} holds:
\begin{equation}\label{eq:mass-transport}
\sum_u f(u,\zero) = \sum_v f(\zero,v) \qquad\text{for any diagonally $\Gamma$-invariant function }f \colon \V^2 \to [0,\infty] .
\end{equation}
By diagonally $\Gamma$-invariant, we mean that $f(\gamma u, \gamma v)=f(u,v)$ for all $u,v \in \V$ and $\gamma \in \Gamma$. We note the well-known fact that, when $G$ is amenable, any transitive group of automorphisms $\Gamma$ is unimodular.

\subsection{Simulating distributions from random bits}
\label{sec:simulation}

We shall use a result about the simulation of a given distribution from unbiased random bits. Let $\mu$ be a distribution on a countable set $\Omega$. A \emph{simulation} of $\mu$ is a pair $\sf S = ({\sf S}^{\text{time}},{\sf S}^{\text{out}})$ of measurable functions ${\sf S}^{\text{time}} \colon \{0,1\}^\N \to \N \cup \{\infty\}$ and ${\sf S}^{\text{out}} \colon \{0,1\}^\N \to \Omega$ with the properties:
\begin{itemize}
 \item If $\omega$ is a sequence of independent unbiased bits, then ${\sf S}^{\text{out}}(\omega)$ has distribution $\mu$.
 \item If ${\sf S}^{\text{time}}(x)=n$ for some $x \in \{0,1\}^\N$ and $n \in \N$, then ${\sf S}^{\text{time}}(x')=n$ and ${\sf S}^{\text{out}}(x')={\sf S}^{\text{out}}(x)$ for any $x' \in \{0,1\}^\N$ which coincides with $x$ on $\{1,\dots,n\}$.
\end{itemize}
The first property says that we can use $\sf S$ to simulate the desired distribution from random bits. The second property may be interpreted as saying that ${\sf S}^{\text{time}}$ is a stopping time and that ${\sf S}^{\text{out}}$ is adapted to the $\sigma$-algebra generated by $\{\omega_i : 1 \le i \le {\sf S}^{\text{time}}\}$ -- that is, the simulation reads one input bit at a time, and once the stopping time is reached, the output is determined only by the bits that have already been read.

The following theorem follows from a result of Knuth and Yao~\cite{knuth1976complexity} (see Theorems~2.1 and~2.2 there and the corollary just after).
\begin{thm}\label{thm:sim}
	Let $Z$ be a discrete random variable.
	There exists a simulation $\sf S$ of $Z$ from independent unbiased bits $\omega$ satisfying that ${\sf S}^{\text{time}}(\omega)<\infty$ almost surely and $\E {\sf S}^{\text{time}}(\omega) \le H(Z) + 2$.
\end{thm}

Knuth and Yao show that the above is in fact optimal in a strong sense (they provide a simulation whose stopping time is stochastically dominated by that of any other simulation).
A version of this theorem was proved in~\cite[Theorem~3]{harvey2006universal} via a more concrete construction.
The results in~\cite{knuth1976complexity,harvey2006universal} also provide explicit exponential bounds on the probability that the simulation uses more than $n$ bits, but we shall not need this.

\section{The cell process}
\label{sec:cell-process}

Recall from Section~\ref{sec:outline} that a cell process is a random sequence $A=(A_1,A_2,\dots)$ of subsets increasing to $\V$ and satisfying that the connected components of each $A_n$ are finite. We may identify a cell process $A$ with the $\N$-valued random field $(\min \{n \ge 1 : v \in A_n \})_{v \in \V}$. In particular, when we say that $A$ is $\Gamma$-invariant or a finitary factor, we mean that this latter process is such.

In this section, we show that finitary cell processes with arbitrarily small entropy exist on any transitive amenable graph.

\begin{prop}\label{prop:cell-process}
	Let $G$ be a transitive amenable graph and let $\epsilon>0$. There exists an \iid\ process~$Y$ of entropy at most $\epsilon$ and a cell process $A$ which is a finitary $\text{Aut}(G)$-factor of $Y$.
\end{prop}

Let us mention that in the special case of $G=\Z^d$, several parts of the argument below can be skipped or simplified, thereby leading to a shorter proof. In the general case, certain technicalities arise which make the proof somewhat longer. The reader may therefore wish to have in mind the case of $G=\Z^d$ on a first reading.

Before giving the proof, let us explain the idea behind it; see Figure~\ref{fig:cell-process} for an illustration.
The main idea of the construction is to use the points of a low-density Bernoulli process to construct Voronoi cells (determined from the Bernoulli process in a finitary manner), which are then used as the cells of $A_1$ (after slightly decreasing the Voronoi cells to ensure that they are well-separated). Using another Bernoulli process of even lower density, we again construct Voronoi cells, which are then used to obtain $A_2$ from $A_1$ by ``filling in'' some of the empty space between the cells of $A_1$, taking care not to connect cells of $A_1$ which are not in the same Voronoi cell (thus ensuring that an infinite cluster is not created). Repeating in this manner, we obtain an increasing sequence $A_1 \subset A_2 \subset \cdots$ of sets (each having only finite cells) as a finitary factor of a small entropy \iid\ process. It will then only remain to show that $A_n$ increases to $\V$. This is where amenability comes into play.

\begin{figure}
	\centering
	\begin{subfigure}[t]{.31\textwidth}
		\centering
		\includegraphics[scale=0.3,trim={1.8cm 12cm 1.8cm 2cm},clip]{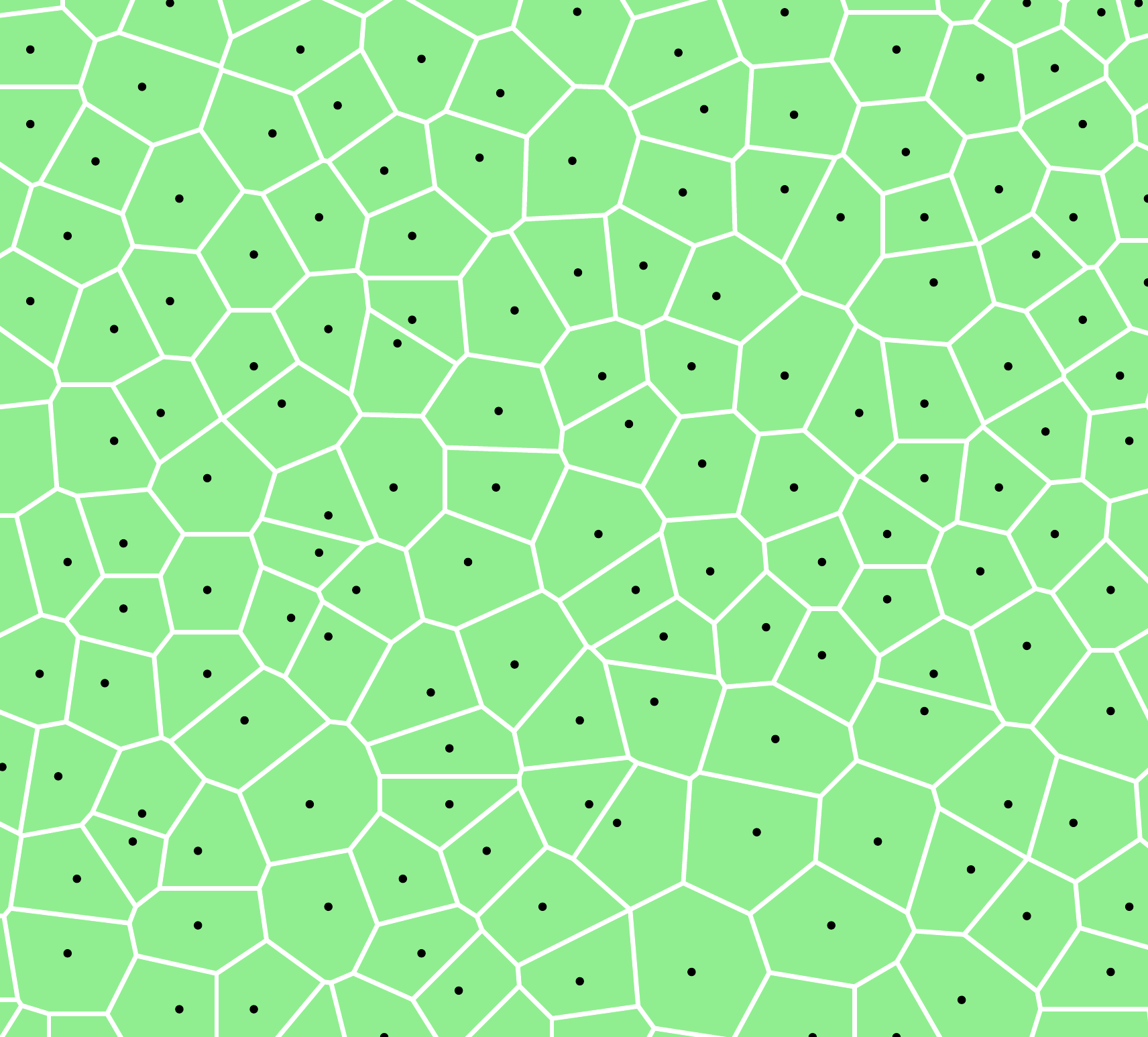}
		\caption{$A_1$}
		\label{fig:cell1}
	\end{subfigure}%
	\begin{subfigure}{15pt}
		\quad
	\end{subfigure}%
	\begin{subfigure}[t]{.31\textwidth}
		\centering
		\includegraphics[scale=0.3,trim={1.8cm 12cm 1.8cm 2cm},clip]{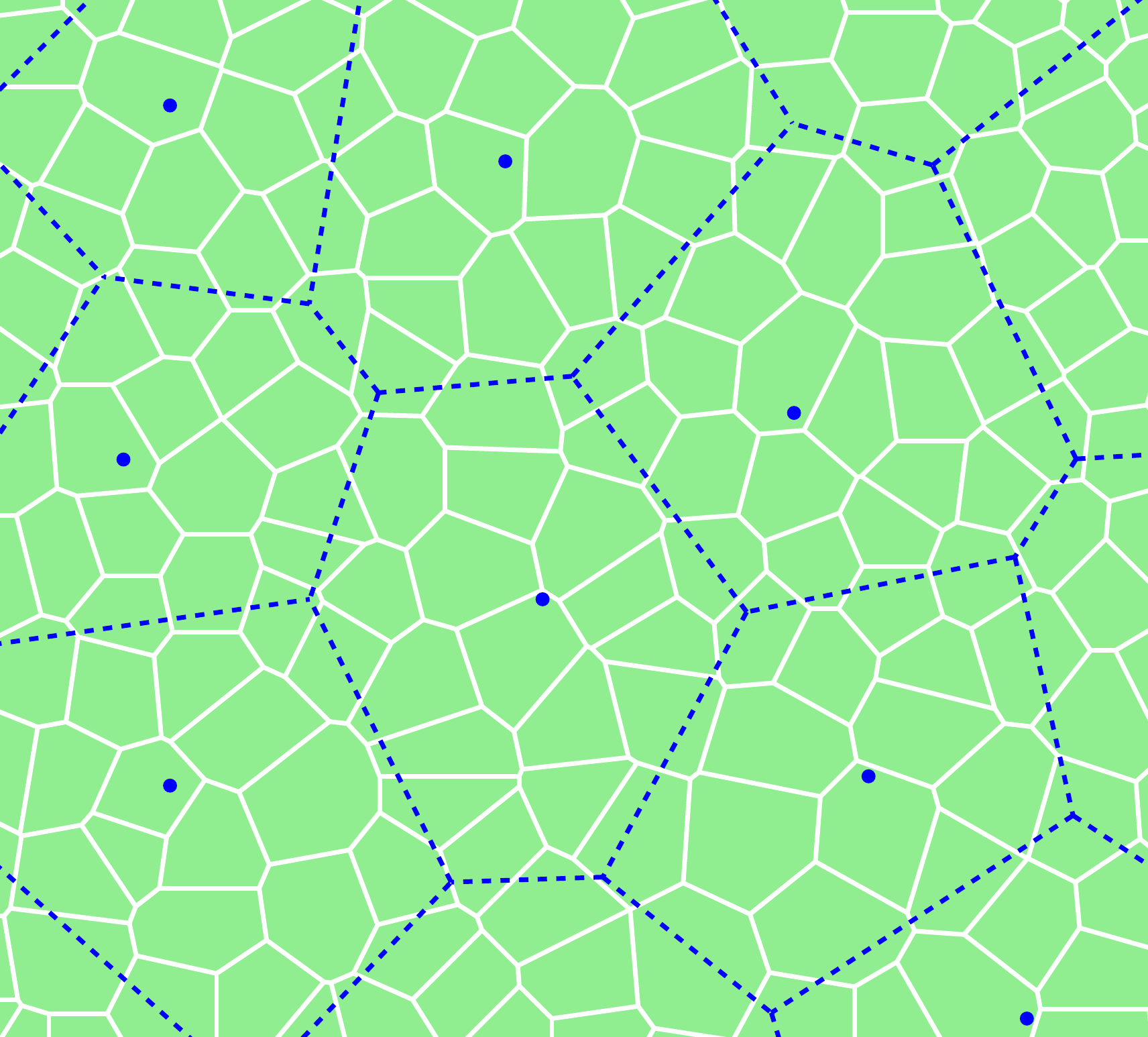}
		\caption{Going from $A_1$ to $A_2$}
		\label{fig:cell2}
	\end{subfigure}%
	\begin{subfigure}{15pt}
		\quad
	\end{subfigure}%
	\begin{subfigure}[t]{.31\textwidth}
		\centering
		\includegraphics[scale=0.3,trim={1.8cm 12cm 1.8cm 2cm},clip]{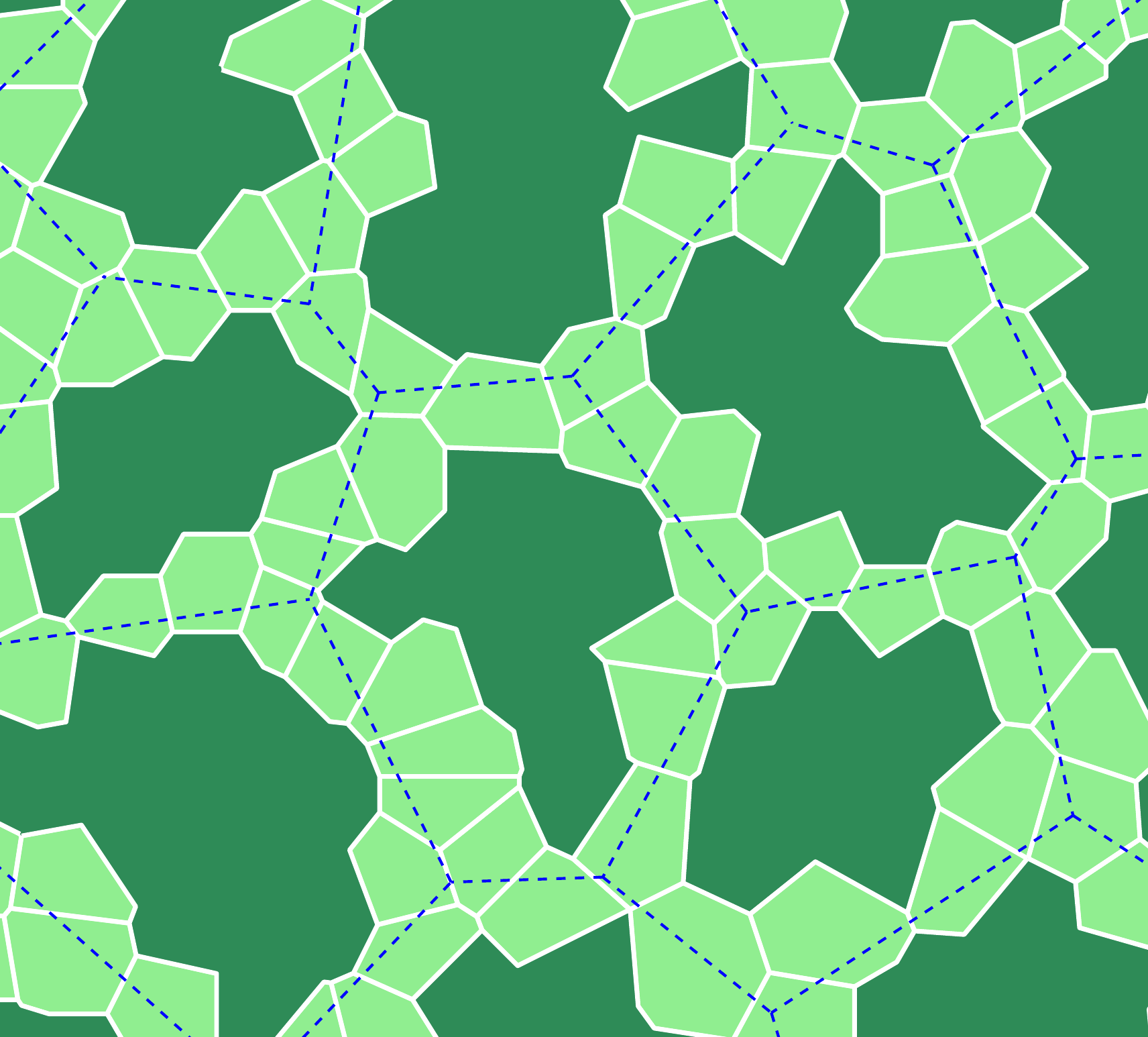}
		\caption{$A_2$}
		\label{fig:cell3}
	\end{subfigure}
	\caption{Constructing the cell process. The cells of $A_1$ are simply the Voronoi cells of a Bernoulli process. To get from $A_1$ to $A_2$, we consider the Voronoi cells of a lower-density Bernoulli process, and ``merge'' cells of $A_1$ which are entirely contained in any such Voronoi cell. Repeating this procedure produces the cell process. The green shade (light and dark) depicts regions belonging to the cell process. The dark green depicts cells of $A_2$ which are not cells of $A_1$.}
	\label{fig:cell-process}
\end{figure}

	To ensure that $A_n$ increases to $\V$, we must be careful in how we define the Voronoi cells. When the graph $G$ has a F{\o}lner sequence consisting of balls (as is the case when $G$ has subexponential growth), the Voronoi cells may be taken with respect to the graph distance in $G$. However, for the general case considered here, we must adapt the usual Voronoi cells to a certain ``metric'' (which is not necessarily a true metric) given by a suitable F{\o}lner sequence. Since we will need this metric to be diagonally $\Gamma$-invariant, we need to choose a F{\o}lner sequence $(F_n)_n$ having the property that each $F_n$ is invariant under the stabilizer of some fixed vertex $\zero \in \V$, i.e., $F_n=\Gamma_{\zero,\zero}F_n$ for all~$n$. It is a simple observation that any graph $G$ of subexponential growth has such a sequence for any $\Gamma$ (since there is a F{\o}lner sequence consisting of balls), and that the Cayley graph $G$ of a finitely generated amenable group $\Gamma$ also has such a sequence (since the stabilizers are trivial).
As it turns out, such a F{\o}lner sequence always exists in an amenable (not necessarily transitive) graph, even when $\Gamma$ is its full automorphism group.

The following lemma (which appears to be new) shows that every amenable graph admits a F{\o}lner sequence consisting of sets which are invariant under the stabilizer of some vertex. The lemma was obtained jointly with Omer Angel.

\begin{lemma}\label{lem:invariant-Folner-sequence}
Let $G$ be an amenable graph, let $\Gamma$ be the full automorphism group of $G$ and let $\zero \in \V$. There exists a F{\o}lner sequence $(F_n)_n$ in $G$ such that $F_n = \Gamma_{\zero,\zero} F_n$ for all $n$.
\end{lemma}

Lemma~\ref{lem:invariant-Folner-sequence} easily follows by applying the following lemma to each set of some F{\o}lner sequence, taking $\Gamma_0$ to be the stabilizer $\Gamma_{\zero,\zero}$ of $\zero$.

\begin{lemma}
Let $G$ be any graph and let $\Gamma_0$ be a group of automorphisms of $G$ under which all orbits of $\V$ are finite. For any finite non-empty set $F \subset \V$, there exists a finite non-empty set $E \subset \V$ such that
\[ \frac{|\partial E|}{|E|} \le \frac{|\partial F|}{|F|} \qquad\text{and}\qquad E = \Gamma_0 E .\]
\end{lemma}
\begin{proof}
We show the existence of the desired $E$ via a probabilistic method -- that is, we choose a random subset $E$ of $\V$ and show that it satisfies the desired properties with positive probability.
Let $\{V_i\}_i$ be the orbits of $\V$ under the action of $\Gamma_0$.
Thus, $\{V_i\}_i$ is a partition of $\V$ such that each $V_i$ is finite (by assumption) and satisfies $V_i = \Gamma_0 V_i$.

Let $U$ be a uniform random variable in $[0,1]$ and define
\[ E := \bigcup_{i : U \le p_i} V_i , \qquad\text{where } p_i := \frac{|F \cap V_i|}{|V_i|} .\]
Thus, each orbit $V_i$ is included in $E$ with probability $p_i$, where the choices for different $i$ are positively correlated through the use of the common variable $U$. Then
\[ \E|E| = \sum_i \Pr(U \le p_i) \cdot |V_i| = \sum_i p_i |V_i| = \sum_i |F \cap V_i| = |F| ,\]
and
\[ \E|\partial E| = \sum_{i,j} \Pr(p_j < U \le p_i) \cdot |\partial(V_i,V_j)| = \sum_{i,j} \max\{p_i-p_j,0\} \cdot |\partial(V_i,V_j)| ,\]
where $\partial(U,V)$ denotes the set of edges between two disjoint sets $U$ and $V$. Let us show that each term in the second sum is at most $|\partial(F \cap V_i, F^c \cap V_j)|$, i.e.,
\[ p_i-p_j \le \frac{|\partial(F \cap V_i, F^c \cap V_j)|}{|\partial(V_i,V_j)|} \qquad\text{whenever }p_i>p_j\text{ and }\partial(V_i,V_j) \neq \emptyset .\]
To see this, note that the right-hand side is the probability that an edge $e=\{u,v\}$ that is uniformly chosen from $\partial(V_i,V_j)$ belongs to $\partial(F \cap V_i, F^c \cap V_j)$, or equivalently, with the convention that $u \in V_i$ and $v \in V_j$, that $u \in F$ and $v \notin F$. Since this probability is at least $\Pr(u \in F)-\Pr(v \in F)$, it suffices to show that $u$ and $v$ are uniformly distributed in $V_i$ and $V_j$, respectively. This follows from the observation that the bipartite graph $(V_i \cup V_j, \partial(V_i,V_j))$ is biregular -- each vertex in $V_i$ is adjacent to the same number of vertices in $V_j$, and similarly, each vertex in $V_j$ is adjacent to the same number of vertices in $V_i$. Indeed, for any $u,v \in V_i$ and $\gamma \in \Gamma_0$ such that $\gamma u = v$, the mapping $w \mapsto \gamma w$ defines a bijection between $N(u) \cap V_j$ and $N(v) \cap V_j$.

We thus conclude that
\[ \E|\partial E| \le \sum_{i,j} |\partial(F \cap V_i, F^c \cap V_j)| = |\partial F| = \alpha \cdot \E|E| ,\]
where $\alpha := |\partial F|/|F|$. Thus, $\alpha |E|-|\partial E|$ is a random variable with non-negative expectation. Since $\alpha |E|-|\partial E|$ is zero when $E$ is empty, conditioned on $E \neq \emptyset$, we still have that $\alpha |E|-|\partial E|$ has non-negative expectation. In particular, there is positive probability that $E\neq \emptyset$ and $\alpha |E| \ge |\partial E|$. Finally, since $E$ is finite and satisfies $E=\Gamma_0 E$ almost surely, we see that $E$ satisfies the desired properties with positive probability.
\end{proof}

Suppose that $(F_n)_n$ is a F{\o}lner sequence guaranteed by Lemma~\ref{lem:invariant-Folner-sequence}, i.e., $F_n = \Gamma_{\zero,\zero} F_n$ for all $n$, and further suppose that $F_1 \subsetneq F_2 \subsetneq \cdots$ and $F_1 \cup F_2 \cup \dots = \V$ (there is clearly no loss in generality in doing so).
Recall the definition of $\Gamma_{u,v}$ from Section~\ref{sec:mass-transport}.
For $u,v \in \V$, define
	\[ \rho(u,v) := \min \big\{ n \ge 1 : v \in \Gamma_{\zero,u} F_n \big\} .\]
	Using that $\Gamma_{\zero,\gamma u} = \gamma \Gamma_{\zero,u}$ for any $\gamma \in \Gamma$, one easily checks that $\rho$ is diagonally $\Gamma$-invariant, i.e., $\rho(\gamma u, \gamma v)=\rho(u,v)$ for all $\gamma \in \Gamma$. We stress that $\rho$ is not necessarily symmetric in that $\rho(u,v)$ may not equal $\rho(v,u)$.
In particular, we do not claim that $\rho$ is a metric. Nevertheless, we still think of $\rho(u,v)$ as a measure of distance from $v$ to $u$. One nice property of $\rho$ that is easily verifiable and which will be important is that, for any sequence of pairs of vertices $(u_i,v_i)_{i=1}^\infty$, we have
\begin{equation}\label{eq:rho-convergence}
\rho(u_i,v_i) \to \infty\text{ as }i\to\infty \qquad\text{if and only if}\qquad \dist(u_i,v_i) \to \infty\text{ as }i\to\infty .
\end{equation}

The fact that $\rho$ may not be symmetric presents a certain challenge in the proof, for which we require the following lemma to address.
We note that $\rho$ is indeed symmetric when the F{\o}lner sequence $(F_n)_n$ consists of balls, and that it is nearly symmetric when $G$ is a Cayley graph of $\Gamma$ in which case switching the roles of $u$ and $v$ in $\rho(u,v)$ has the same effect as replacing each $F_n$ with $F_n^{-1}$. Accordingly, in these cases, it is immediate that the two ``$\rho$-balls'' of radius $n$ around~$\zero$, $\{ v : \rho(\zero,v) \le n\}$ and $\{ u : \rho(u,\zero) \le n \}$, have the same size, namely $|F_n|$.
The following lemma shows that this is in fact always true in our setting.

\begin{lemma}\label{lem:size-of-inverted-set}
Let $G$ be a transitive amenable graph and let $\Gamma$ be the full automorphism group of $G$.
Let $F \subset \V$ be invariant under the stabilizer of some vertex $\zero$, i.e., $\Gamma_{\zero,\zero}F=F$. Then
\[ |\{ u \in \V : \zero \in \Gamma_{\zero,u} F \}| = |F| .\]
\end{lemma}
\begin{proof}
Define $f \colon \V^2 \to [0,1]$ by
\[ f(u,v) := \1_{\{v \in \Gamma_{\zero,u} F\}} .\]
Since $\Gamma_{\zero,\gamma u} = \gamma \Gamma_{\zero,u}$ for any $\gamma \in \Gamma$, it follows that $f$ is diagonally $\Gamma$-invariant.
Thus, by the mass-transport principle~\eqref{eq:mass-transport},
\[ |\{ u \in \V : \zero \in \Gamma_{\zero,u} F \}| = \sum_u f(u,\zero) = \sum_v f(\zero,v) = |\Gamma_{\zero,\zero} F| = |F| . \qedhere \]
\end{proof}

We are now ready to give the proof of Proposition~\ref{prop:cell-process}.

\begin{proof}[Proof of Proposition~\ref{prop:cell-process}]
	Let $(\epsilon_n)$ be a sequence to be chosen later which satisfies that $0 \le \epsilon_n \le \epsilon 2^{-n}$.
	We shall construct a cell process $A$ as a finitary factor of the \iid\ process $Y=(Y_v)_{v \in \V}$ in which $Y_v=(Y_{v,n})_{n \ge 1}$ are independent random variables with $Y_{v,n} \sim \text{Ber}(\epsilon_n)$.	
	The entropy of $Y$ can be made arbitrary small, since
	\[ h(Y) = H(Y_v) = \sum_{n=1}^\infty H(Y_{v,n}) = \sum_{n=1}^\infty \left( \epsilon_n \log \tfrac{1}{\epsilon_n} + (1-\epsilon_n) \log \tfrac{1}{1-\epsilon_n} \right) \le 10\epsilon \log \tfrac{1}{\epsilon} .\]
	
	As explained, the idea of the construction is to use the points in
	\[ U_n := \{ v \in \V : Y_{v,n}=1\} ,\]
	for any given~$n$, to construct Voronoi cells, which are then used to define the cells of $A_n$.	
	Precisely, we define the Voronoi cells of a non-empty set $U \subset \V$ by
	\[ \bar{\cC}_U(u) := \Big\{ v \in \V : \rho(u,v) < \rho(u',v)\text{ for all }u' \in U \setminus \{u\} \Big\} ,\qquad u \in U .\]
	Thus, the Voronoi cell $\bar{\cC}_U(u)$ associated to $u$ consists of all vertices $v \in \V$ which are closer (in the distance measured by $\rho$) to $u$ than to any other $u' \in U$. In particular, Voronoi cells associated to different vertices in $U$ are disjoint, but they are not necessarily separated (they could be adjacent to one another). We therefore define modified Voronoi cells by slightly shrinking the sets $\bar{\cC}_U(u)$. Precisely, we define
	\[ \cC_U(u) := (\bar{\cC}_U(u))^{-1} .\]
	Using that the Voronoi cells are disjoint, it is straightforward to check that $\dist(\cC_U(u),\cC_U(u')) > 1$ for distinct $u,u' \in U$. We note that $\cC_U(u)$ (in fact, already $\bar{\cC}_U(u)$) may be empty and need not be connected.

	Before proceeding with the construction of the cell process, let us first show that the Voronoi cells of $U$ are almost surely finite whenever $U$ is the set of points of an \iid\ Bernoulli process. To this end, it suffices to show that the probability that $\bar{\cC}_U(\zero)$ intersects $F_n \setminus F_{n-1}$ is summable over~$n$. Indeed, since a fixed vertex $v \in F_n \setminus F_{n-1}$ belongs to $\bar{\cC}_U(\zero)$ only if $U \setminus \{\zero\}$ contains no element of $\{ u : \rho(u,v) \le n \}$, it follows from Lemma~\ref{lem:size-of-inverted-set} that the probability of this is at most $p^{|F_n|-1}$, where $p$ is the density of the Bernoulli process. Since $|F_n| \ge n$, we see that $|F_n| \cdot p^{|F_n|-1}$ is summable, and hence that the $\bar{\cC}_U(\zero)$ is almost surely finite.
	
	We now turn to the construction of the cell process $A$.
	The first level set in the cell process is simply taken to be the vertices in a modified Voronoi cell of $U_1$, i.e.,
	\[ A_1 := \bigcup_{u \in U_1} \cC_{U_1}(u) .\]
	Since the Voronio cells of $U_1$ are almost surely finite, we see that $\cC_{U_1}(u)$ is almost surely finite for all $u \in U_1$. Since $\dist(\cC_{U_1}(u),\cC_{U_1}(u')) > 1$ for distinct $u,u' \in U_1$, it follows that all connected components of $A_1$ are almost surely finite.
	
	Suppose now that, for some $n \ge 1$, $A_n$ has been defined in such a way that all connected components of $A_n$ are almost surely finite, and let us now define $A_{n+1}$.
	Let $A'_{n+1}$ be the union of the modified Voronoi cells of $U_{n+1}$, i.e.,
	\[ A'_{n+1} := \bigcup_{u \in U_{n+1}} \cC_{U_{n+1}}(u) ,\]
	and recall that (as for $A_1$) all connected components of $A'_{n+1}$ are almost surely finite. Intuitively, we would like to obtain $A_{n+1}$ from $A_n$ by adding $A'_{n+1}$. However, this might create infinite clusters, and we must take care to avoid this by instead only adding a suitable subset of $A'_{n+1}$. It will suffice to slightly increase the ``forbidden region'' $(A'_{n+1})^c$ as follows: let $D_{n+1}$ denote the union of the connected components of $A_n$ that intersect $(A'_{n+1})^c$, and add $D_{n+1}^{+1}$ to the forbidden region. Precisely, we define
	\[ A_{n+1} := A_n \cup (A'_{n+1} \setminus D_{n+1}^{+1}) .\]
	It is straightforward to check that all connected components of $A_{n+1}$ are almost surely finite. Assuming that $A_1 \cup A_2 \cup \dots = \V$ almost surely, it is also easy to check using~\eqref{eq:rho-convergence} that $A$ is a finitary $\text{Aut}(G)$-factor of $Y$, which would complete the proof of the proposition.
	
	It remains only to show that $A_1 \cup A_2 \cup \dots = \V$ almost surely.
	By $\text{Aut}(G)$-invariance, this is equivalent to the fact that $\Pr(\zero \in A_n) \to 1$ as $n \to \infty$.
	For $n \ge 1$, let $B_n$ denote the connected component of $\zero$ in $A_n \cup \{\zero\}$. For $n \ge 2$, define the event
	\[ E_n := \Big\{ B_{n-1} \subset \cC_{U_n}(u)\text{ for some }u \in U_n \Big\} .\]
	Note that
	\[ \{ \zero \in A_n \setminus A_{n-1} \} = E_n \cap \{ \zero \notin A_{n-1} \} .\]
	Thus, it suffices to show that
	\[ \Pr(E_n \mid \zero \notin A_{n-1}) \ge c \qquad\text{for some }c>0\text{ and all }n \ge 2 .\]
	As we now show, this holds when $\epsilon_n$ is suitably chosen.
	Since $B_{n-1}$ is almost surely finite, there exists a sufficiently large $r_n$ so that
	\[ \Pr(B_{n-1} \subset \Lambda_{r_n-1} \mid \zero \notin A_{n-1}) \ge \tfrac12 ,\]
	where $\Lambda_r := \Lambda_r(\zero)$.
	Since $(F_s)_s$ is a F{\o}lner sequence, there exists $s_n$ sufficiently large so that
	\begin{equation}\label{eq:choice-of-s_n}
	|F_{s_n}| \ge \frac{2^n}{\epsilon} \qquad\text{and}\qquad \frac{|\partial F_{s_n}|}{|F_{s_n}|} \le \frac{1}{2|\Lambda_{r_n}|} .
	\end{equation}
	Set $\epsilon_n := |F_{s_n}|^{-1}$.
	Then, noting that $A_{n-1}$ (and thus also $B_{n-1}$) is independent of $U_n$,
	\begin{align*}
	 \Pr(E_n \mid \zero \notin A_{n-1})
	  &\ge \Pr\big(B_{n-1} \subset \Lambda_{r_n-1}\text{ and }\Lambda_{r_n-1} \subset \cC_{U_n}(u)\text{ for some }u \in U_n \mid \zero \notin A_{n-1}\big) \\
	  &= \Pr\big(B_{n-1} \subset \Lambda_{r_n-1} \mid \zero \notin A_{n-1}\big) \cdot \Pr\big(\Lambda_{r_n-1} \subset \cC_{U_n}(u)\text{ for some }u \in U_n\big) .
	\end{align*}
	Since the first term on the right-hand side is at least $\frac12$ by the choice of $r_n$, it remains to show that, for some constant $c>0$ which does not depend on $n$, we have
	\[ \Pr\big(\Lambda_{r_n-1} \subset \cC_{U_n}(u)\text{ for some }u \in U_n\big) \ge c .\]
	
	Let us first see how to show this when $F_{s_n}$ is a ball, say $\Lambda_\ell$. In this case, it is not hard to see that the event in question occurs when $U_n$ has a unique point in $\Lambda_{\ell-r_n}$ and no other point in $\Lambda_{\ell+r_n}$, so that
\begin{align*}
	&\Pr\big(\Lambda_{r_n-1} \subset \cC_{U_n}(u)\text{ for some }u \in U_n\big)\\
		  &\qquad \qquad \ge \Pr\big(|U_n \cap \Lambda_{\ell-r_n}|=|U_n \cap \Lambda_{\ell+r_n}|=1\big) \\
	  &\qquad \qquad = \Pr\big(\text{Ber}(|\Lambda_{\ell-r_n}|,\epsilon_n)=1\big) \cdot \Pr\big(\text{Ber}(|\Lambda_{\ell+r_n}\setminus \Lambda_{\ell-r_n}|,\epsilon_n)=0\big) \ge c .
	\end{align*}
		
	We now handle the general case in more detail. Set $U:=U_n$.	
	Our goal is to bound from below the probability that $\Lambda_{r_n-1} \subset \cC_U(u)$ for some $u \in U$.
	To this end, we first find a simple condition that implies the occurrence of this event.
	For a set $F \subset \V$, denote
	\[ M(F) := \{ u \in \V : \zero \in \Gamma_{\zero,u} F \} .\]
	Set $r:=r_n$ and $s := s_n$.
	Let us show that
	\begin{equation}\label{eq:M-cond}
	|U \cap M(F_s^{-r})|=|U \cap M(F_s^{+r})|=1 \quad\implies\quad \Lambda_{r-1} \subset \cC_U(u)\text{ for some }u \in U .
	\end{equation}	
	Suppose that the left-hand side holds. Let us show that $\Lambda_{r-1} \subset \cC_U(u)$, where $u$ is the unique element in $U \cap M(F_s^{-r})$. By the definition of $\cC_U(u)$, this is equivalent to $\Lambda_r \subset \bar{\cC}_U(u)$. Recalling the definition of $\bar{\cC}_U(u)$, we see that we must show that $\rho(u,w) < \rho(u',w)$ for all $u' \in U \setminus \{u\}$ and $w \in \Lambda_r$.
	Let $u' \in U \setminus \{u\}$ and $w \in \Lambda_r$.
	It suffices to show that $\rho(u,w) \le s$ and $\rho(u'
	,w)>s$.

	Towards showing this, we first note that $(\gamma V)^{+1}=\gamma (V^{+1})$ and $(\gamma V)^{-1}=\gamma (V^{-1})$ for any $\gamma \in \Gamma$ and $V \subset \V$, due to the fact that $\gamma$ acts by an automorphism of $G$. In particular, $(\Gamma_{\zero,u} V)^{+r} = \Gamma_{\zero,u} (V^{+r})$ and $(\Gamma_{\zero,u} V)^{-r} = \Gamma_{\zero,u} (V^{-r})$, and we may drop the parenthesis when writing such terms.
	
	Let us now show that $\rho(u,w) \le s$.
	Since $u \in M(F_s^{-r})$, we have that $\zero \in \Gamma_{\zero,u} F_s^{-r}$, or equivalently, $\Lambda_r \subset \Gamma_{\zero,u} F_s$. Since $w \in \Lambda_r$, it follows that $\rho(u,w) \le s$.
	Next, we show that $\rho(u',w)>s$.
	Note that $u' \notin M(F_s^{+r})$ since $u \in M(F_s^{-r}) \subset M(F_s^{+r})$. Thus, $\zero \notin \Gamma_{\zero,u'} F_s^{+r}$, or equivalently, $\Lambda_r \cap \Gamma_{\zero, u'} F_s = \emptyset$. Thus, $w \notin \Gamma_{\zero,u'} F_s$ and, using that $F_1,F_2,\dots,F_{s-1} \subset F_s$, it follows that $\rho(u',w)>s$.

	Using~\eqref{eq:M-cond}, we obtain
\begin{align*}
	&\Pr\big(\Lambda_{r_n-1} \subset \cC_{U_n}(u)\text{ for some }u \in U_n\big)\\
		  &\qquad \qquad \ge \Pr\big(|U_n \cap M(F_{s_n}^{-r_n})|=|U_n \cap M(F_{s_n}^{+r_n})|=1\big) \\
	  &\qquad \qquad = \Pr\big(\text{Ber}(|M(F_{s_n}^{-r_n})|,\epsilon_n)=1\big) \cdot \Pr\big(\text{Ber}(|M(F_{s_n}^{+r_n})|-|M(F_{s_n}^{-r_n})|,\epsilon_n)=0\big) .
	\end{align*}
	By Lemma~\ref{lem:size-of-inverted-set}, we have that
	\[ |M(F_{s_n}^{-r_n})| = |F_{s_n}^{-r_n}| \qquad\text{and}\qquad |M(F_{s_n}^{+r_n})| = |F_{s_n}^{+r_n}| .\]
	Finally, by~\eqref{eq:choice-of-s_n},
	\[ |F_{s_n}^{-r_n}| \ge |F_{s_n}| - |\partial F_{s_n}| \cdot |\Lambda_{r_n}| \ge \tfrac12 |F_{s_n}| \qquad\text{and}\qquad |F_{s_n}^{+r_n}| \le |F_{s_n}| + |\partial F_{s_n}| \cdot |\Lambda_{r_n}| \le \tfrac32 |F_{s_n}| ,\]
	so that, by standard estimates for Bernoulli random variables, both probabilities in question are bounded below by a positive constant.
\end{proof}

The above proposition established the existence of a finitary cell process $A$. In particular, $A_n$ is an invariant set which has high density when $n$ is large. The following proposition shows that the clusters of a dense invariant set typically have relatively small boundary.

\begin{lemma}\label{lem:small-boundary}
Let $G$ be a transitive graph of degree $d$ and let $\Gamma$ be a transitive unimodular group of automorphisms of $G$.
Let $B \subset \V$ be a random set with no infinite clusters and whose distribution is $\Gamma$-invariant. Let $\cC_v$ denote the cluster of $v$ in $B$. Then, for any $\delta>0$,
\[ \Pr\big( |\partial \cC_\zero| \ge \delta |\cC_\zero| \big) \le (\tfrac d\delta + 1) \cdot \Pr(\zero \notin B) .\]
\end{lemma}
\begin{proof}	
The proof uses the mass-transport principle.
Define $\psi \colon \V^2 \to [0,1]$ by
\[ \psi(u,v) := \begin{cases}
 \frac{|N(u) \cap \cC_v|}{|\cC_v|} &\text{if }u \notin B,~v \in B\\
 0 &\text{otherwise}
\end{cases} .\]
Note that, almost surely,
\[ \sum_u \psi(u,\zero) = \frac{\1_{\zero \in B}}{|\cC_\zero|} \cdot \sum_{u \notin B} |N(u) \cap \cC_\zero| = \frac{|\partial \cC_\zero|}{|\cC_\zero|} \cdot \1_{\zero \in B} \]
and
\[ \sum_v \psi(\zero,v) = \1_{\zero \notin B} \cdot \sum_{v \in B} \frac{|N(\zero) \cap \cC_v|}{|\cC_v|} = |N(\zero) \cap B| \cdot \1_{\zero \notin B} \le d \cdot \1_{\zero \notin B} .\]
The $\Gamma$-invariance of $B$ implies that $f(u,v) := \E \psi(u,v)$ is diagonally $\Gamma$-invariant. Thus, the mass-transport principle~\eqref{eq:mass-transport} yields that
\[ \E\left[ \tfrac{|\partial \cC_\zero|}{|\cC_\zero|} \cdot \1_{\zero \in B} \right] \le d \cdot \Pr(\zero \notin B) .\]
The proposition now follows from Markov's inequality.
\end{proof}

\begin{remark}
A result of H{\"a}ggstr{\"o}m~\cite[Theorem~1.6]{haggstrom1997infinite} states that any automorphism-invariant edge percolation on a $d$-regular tree ($d \ge 3$) with edge-density at least $2/d$ has an infinite cluster with positive probability. In particular, automorphism-invariant cell processes do not exist on such a tree. Moreover, by Lemma~\ref{lem:small-boundary} (see also the closely related~\cite[Theorem~1.2]{benjamini1999group}), we see that $\Gamma$-invariant cell processes do not exist on any transitive unimodular non-amenable graph.
In fact, it is shown in~\cite[Theorem~5.1]{benjamini1999group} that a closed subgroup $\Gamma$ of $\text{Aut}(G)$ is amenable if and only if there is a $\Gamma$-invariant site percolation on $G$ with with no infinite clusters and density arbitrarily close to 1. It follows that a $\Gamma$-invariant cell process on $G$ exists if and only if $\Gamma$ is amenable. The site percolation constructed in~\cite{benjamini1999group} is a factor of an \iid\ process, though it is not finitary.
\end{remark}

\section{Random total orders}
\label{sec:random-orders}

In this section, we construct a random total order on $\V$ that has the following properties:
\begin{itemize}
 \item It is a finitary factor of an \iid\ process with arbitrarily small entropy.
 \item It is supported on total orders having the same order type as $\Z$.
 \item The successor/predecessor of any vertex can be found in a finitary manner.
\end{itemize}

Let us explain these properties.
A random total order on $\V$, or more generally, a random binary relation on $\V$, may be regarded as a random element in $\{0,1\}^{\V^2}$. With this viewpoint, the notion of finitary factor easily applies to such relations. Namely, such a relation is a $\Gamma$-factor of $Y$ if it has the same distribution as $\varphi(Y)$ for some measurable function $\varphi \colon T^\V \to \{0,1\}^{\V^2}$ satisfying that $\varphi(y)_{(u,v)}=\varphi(\gamma y)_{(\gamma u,\gamma v)}$ for all $\gamma \in \Gamma$, $u,v \in \V$ and $y \in T^\V$. Such a factor is finitary if for every $u,v \in \V$ there almost surely exists a finite (random) set $W \subset \V$ such that $\varphi(Y)_{(u,v)}$ is determined by $(Y_w)_{w \in W}$, in the sense that $\varphi(y)_{(u,v)}=\varphi(Y)_{(u,v)}$ for any $y \in T^\V$ which coincides with $Y$ on $W$.

A total order $\le$ on $\V$ has the same order type as $\Z$ if there is an order preserving bijection between the two ordered spaces, i.e., a bijection $f \colon \V \to \Z$ such that $f(u) \le f(v)$ if and only if $u \le v$.
This may be equivalently formulated as saying that $\le$ has no minimum or maximum and that there are finitely many elements between any two elements, i.e., every interval of the form $\{ w \in \V : u \le w \le v \}$ is finite. In particular, in such an order, every vertex $v$ has a successor (an element $w \ge v$ such that $u \ge w$ for all $u \ge v$) and a predecessor (an element $w \le v$ such that $u \le w$ for all $u \le v$).

Given a factor from $Y$ to a random total order $\le$ on $\V$, we say that successors (predecessors) can be found in a finitary manner if for every $u,v \in \V$ there almost surely exists a finite (random) set $W \subset \V$ such that the event that $u$ is the $\le$-successor ($\le$-predecessor) of $v$ is determined by $(Y_w)_{w \in W}$.
We note that, in general, there is no direct relation to the notion of finitary factor: it may be that such a factor is finitary though successors/predecessors cannot be found in a finitary manner, or it may that successors/predecessors can be found in a finitary manner though the factor is not finitary.
On the other hand, for a total order having the order type of $\Z$ almost surely, the second implication is easily seen to hold -- if successors/predecessors can be found in a finitary manner, then the factor is necessarily finitary.

A total order which is a finitary factor of an \iid\ process with infinite entropy is easily obtained from the order induced by uniform random variables in $[0,1]$ assigned to each vertex. It is easy to see that this order almost surely has the same order type as $\mathbb Q$. A total order (also with the order type of $\mathbb Q$) which is a finitary factor (with exponential tails on the coding radius) of an \iid\ process with finite entropy was constructed in~\cite{harel2018finitary} for any quasi-transitive graph satisfying a geometric condition similar to~\eqref{eq:sphere-condition}.
The application in~\cite{harel2018finitary} did not require the \iid\ process to have arbitrarily small entropy and so this was not stated there, though it easily follows from the proof there that this is possible. Since the proof is short, we give it here. The following is essentially a reformulation of~\cite[Lemma~18]{harel2018finitary} for our situation.

\begin{lemma}\label{lem:finiteordering}
Let $G$ be a transitive non-empty graph satisfying~\eqref{eq:sphere-condition}. For any $0<\epsilon \le \frac12$ there exists a total order on $\V$ which is a finitary $\text{Aut}(G)$-factor of an \iid\ Bernoulli process with density~$\epsilon$.
\end{lemma}
\begin{proof}
	Let $\eta=(\eta_v)_{v \in \V}$ be an \iid\ Bernoulli process with density~$\epsilon$.
	For any $v \in \V$, define $Z_v=(Z_{v,n})_{n \ge 0} \in \{0,1,\dots\}^{\{0,1,\dots\}}$ by
\[
Z_{v,n} := \sum_{u \in \V : \dist(u,v)=n} \eta_u.
\]
Define a relation $\le$ on $\V$ in which $u \le v$ if and only if $Z_u \preceq Z_v$, where $\preceq$ denotes the lexicographical order on $\{0,1,\dots\}^{\{0,1,\dots\}}$. Then $\le$ is clearly a $\text{Aut}(G)$-factor of $\eta$.

It remains to show that $\le$ is almost surely a total order on $\V$ and that the factor is finitary.
Since $\le$ is clearly a preorder, to show that it is a total order, it suffices to show that $\Pr(Z_u=Z_v)=0$ for distinct $u,v \in \V$.
It then follows from the definition of the lexicographical order that the factor is finitary.

	Fix $u,v \in \V$ distinct and consider the event
	\[ E_n := \bigcap_{i=1}^n \{ Z_{u,i} = Z_{v,i} \} .\]
	Since $\Pr(E_n) \to \Pr(Z_u=Z_v)$ as $n \to \infty$, it suffices to show that $\Pr(E_n \mid E_{n-1}) \le 1-\epsilon$ for all $n \ge 1$.
	By~\eqref{eq:sphere-condition} and the assumption that the graph is non-empty, we have $\Lambda_n(u) \setminus \Lambda_{n-1}(u) \not\subset \Lambda_n(v)$, as otherwise $\Lambda_{3n}(u) \subset \Lambda_{3n}(v)$, which in turn implies that $\Lambda_{3n}(u) = \Lambda_{3n}(v)$ by transitivity.
	Thus, there exists some $w_n \in \Lambda_n(u) \setminus (\Lambda_{n-1}(u) \cup \Lambda_n(v))$. Then
	\[ \Pr\big(E_n \mid \eta_{\V \setminus \{w_n\}}\big) \le \max_{k \in \Z} \Pr(\eta_{w_n}=k) = \max\{\epsilon,1-\epsilon\} = 1-\epsilon .\]
	Since $E_{n-1}$ is measurable with respect to $\eta_{\V \setminus \{w_n\}}$, it follows that $\Pr(E_n \mid E_{n-1}) \le 1-\epsilon$.
\end{proof}

Using Lemma~\ref{lem:finiteordering} and the cell processes constructed in the previous section, we are able to construct a total order satisfying all three properties described above.

\begin{lemma}\label{lem:total-order}
Let $G$ be a transitive amenable non-empty graph satisfying~\eqref{eq:sphere-condition} and let $\epsilon>0$. Then there exists an \iid\ process $Y$ with entropy at most $\epsilon$, and a random total order $\le$ on $\V$ which almost surely has the same order type as $\Z$, such that $\le$ is a finitary $\text{Aut}(G)$-factor of $Y$ for which successors/predecessors can be found in a finitary manner.
\end{lemma}

\begin{proof}
By Lemma~\ref{lem:finiteordering}, there exists a total order $\preceq$ on $\V$ that is a finitary factor of an \iid\ process~$Y$ having entropy at most $\epsilon$. By Proposition~\ref{prop:cell-process}, there exist a cell process $A$ that is a finitary factor of an \iid\ process $Y'$ (which we take to be independent of $Y$) having entropy at most $\epsilon$. We construct the required total order $\le$ on $\V$ as a finitary factor of $(Y,Y')$.

The idea is to use $\preceq$ to order the sites within the cells given by $A$. More precisely, we will define an increasing sequence of partial orders $\le_1 \,\subset\, \le_2 \,\subset\, \dots$ such that each $\le_n$ induces a total order on every cell of $A_n$. Since $\bigcup_n A_n = \V$, this will produce a total order $\le$ given by the union $\bigcup_n \le_n$, which we will show has the desired properties. 

Precisely, we define $\le_1$ to be the relation in which $u \le_1 v$ whenever $u$ and $v$ belong to the same cell of $A_1$ and satisfy that $u \preceq v$. Then $\le_1$ is clearly a partial order that induces a total order on any cell of $A_1$. In fact, it is the union of these total orders on the cells of $A_1$ (that is, it only compares vertices that belong to the same cell).

Next, suppose we have defined the partial order $\le_{n-1}$ so that it is a union of total orders on the cells on $A_{n-1}$, and let us define $\le_n$. Consider a cell $C$ of $A_n$ and let $D:=C \cap A_{n-1} = D_1 \cup \dots \cup D_k$ be the union of the cells $D_1,\dots,D_n$ of $A_{n-1}$ that are contained in $C$. We define $\le_n$ in such a way that $D \le_n C \setminus D$ by requiring that $u \le_n v$ whenever $u \in D$ and $v \in C \setminus D$. To obtain a total order on~$C$, it remains to order the vertices in $D$ and the vertices in $C \setminus D$. The latter is ordered by defining $u \le_n v$ whenever $u,v \in C \setminus D$ and $u \preceq v$. The former is ordered by giving an order to the cells $D_1,\dots,D_k$ and using the $\le_{n-1}$ order within each cell -- that is, we require that $\le_n$ coincides with $\le_{n-1}$ on each cell $D_i$, and that either $D_i \le_n D_j$ or $D_j \le_n D_i$ for any two cells $D_i$ and $D_j$. Finally, the order of the cells is determined by requiring that $D_i \le_n D_j$ whenever $\min_{\preceq} D_i \preceq \min_{\preceq} D_j$. That is, the $i$-th cell precedes the $j$-cell in $\le_n$ if and only if the $\preceq$-minimal element in the $i$-th cell is $\preceq$-smaller than the $\preceq$-minimal element in the $j$-th cell. It is straightforward that $\le_n$ extends $\le_{n-1}$ and that $\le_n$ is a union of total orders on the cells of $A_n$.

We have thus obtained partial orders $\le_1,\le_2,\dots$ such that, for each $n$, $\le_n$ extends $\le_{n-1}$ and is a union of total orders on the cells of $A_n$. To show that $\le$ has the order type of $\Z$, it remains to show that, almost surely, every $\le$-interval is finite and there is no $\le$-minimum and no $\le$-maximum. 
It follows from the construction that if $u$ is the $\le_n$-successor of~$v$, then it is also its $\le_{n+1}$-successor. Thus, to conclude that every $\le$-interval is finite, it suffices to show that, for every $u,v \in \V$ having $u \le v$, there exists $n$ such that $u \le_n v$ and the interval $[u,v]_{\le_n}$ is finite. Indeed, since $\le_n$ only compares vertices within the same cell of $A_n$ and since all such cells are finite, all $\le_n$-intervals are finite.
Finally, no minimum or maximum can exist as this would contradict the invariance of~$\le$.

We have thus established that $\le$ almost surely has the same order type as $\Z$. It is straightforward from the fact that the $\le_n$-successor of a vertex $u$ (if it exists) is also the $\le_{n+1}$-successor of $u$, that the constructed factor from $(Y,Y')$ to $\le$ has the property that successors/predecessors can be found in a finitary manner. As mentioned in the beginning of the section, this implies that the factor is also finitary.
\end{proof}

\section{The finitary coding}
\label{sec:coding}

In this section, we construct a finitary coding for finitely dependent processes. We present the details of the proof of Theorem~\ref{thm:main}. Theorem~\ref{thm:main-no-entropy} may be proved in a similar manner (see Remark~\ref{rem:other-thm-proof} in Section~\ref{sec:conclusion}).

Let $X$ be a $\Gamma$-invariant finitely dependent process taking values in a finite set $S$.
Recall from the proof outline in Section~\ref{sec:outline} that we shall construct a finitary factor from an \iid\ process $Y=(Y^{\text{bits}},Y^{\text{cell}},Y^{\text{ord}})$ to $X$.
Recall also that $A$ will be a cell process that is a finitary factor of $Y^{\text{cell}}$ and that $\le$ will be a random total order on $\V$ that is a finitary factor of $Y^{\text{ord}}$, has the order type of $\Z$, and for which successors/predecessors can be found in a finitary manner. Given the cell process $A$ and the total order $\le$, we will use the additional randomness in $Y^{\text{bits}}$ to construct a realization of $X$.

\subsection{Choosing the parameters}
Fix $\epsilon>0$.
We shall choose the \iid\ processes $Y^{\text{bits}},Y^{\text{cell}},Y^{\text{ord}}$ to satisfy
$h(Y^{\text{bits}})<h(X)+5\epsilon$, $h(Y^{\text{cell}}) \le \epsilon$ and $h(Y^{\text{ord}}) \le \epsilon$ so that $Y$ has entropy
\[ h(Y) < h(X)+7\epsilon .\]

We let $Y^{\text{bits}}$ be any \iid\ process in which $Y^{\text{bits}}_v$ is a random number of random bits satisfying
\begin{equation}\label{eq:entropy-choice}
H(Y^{\text{bits}}_v)<h(X)+5\epsilon \qquad\text{and}\qquad \E|Y^{\text{bits}}_v| > h(X) + 3\epsilon .
\end{equation}
Recall that, by a random number of random bits, we mean a random variable $W$ taking values in $\{0,1\}^*$, the set of finite words over $\{0,1\}$, and having the property that, conditioned on the length $|W|$ of the word, $W$ is uniformly distributed on $\{0,1\}^{|W|}$. 
The desired random word can be obtained by taking $W$ to be the empty word with probability $p$ or a uniformly chosen sequence in $\{0,1\}^m$ with probability $1-p$, for some suitably chosen $m \ge 1$ and $0 \le p < 1$. Indeed, in this case, $H(|W|)=-p\log p - (1-p)\log (1-p)$ and $\E|W|=pm$ so that $H(|W|) \to 0$ and $\E|W| \to h(X)+4\epsilon$ as $m \to \infty$ when $p=\frac1m(h(X) + 4\epsilon)$. Since the entropy and length of $W$ are related via $H(W)=\E|W|+H(|W|)$, we see that~\eqref{eq:entropy-choice} holds when $m$ is sufficiently large.

Let $\delta>0$ be as in~\eqref{eq:entropy-via-boundary}. By decreasing $\delta$, we may additionally assume that
\begin{equation}\label{eq:delta}
2 \delta < \epsilon .
\end{equation}
Recall from Section~\ref{sec:outline} that we may assume without loss of generality that $X$ is 1-dependent.
By Proposition~\ref{prop:cell-process}, there exists a cell process $A$ and an \iid\ process $Y^{\text{cell}}$ of entropy at most $\epsilon$ such that $A$ is a finitary factor of $Y^{\text{cell}}$.
Since $\Pr(\zero \in A_n) \to 1$ as $n \to \infty$, Lemma~\ref{lem:small-boundary} implies that $\Pr(|\partial A_n(\zero)| \ge \delta |A_n(\zero)|) \to 0$ as $n \to \infty$, where $A_n(\zero)$ is the cell of $\zero$ in $A_n$. Thus, by replacing $(A_1,A_2,\dots)$ with $(A_m,A_{m+1},\dots)$ for some large $m$, we may assume that $\Pr(|\partial A_1(\zero)| \ge \delta |A_1(\zero)|)$ is arbitrary small. Specifically, we require that
\begin{equation}\label{eq:small-cluster-boundary}
\Pr\Big( |\partial A_1(\zero)| \ge \delta |A_1(\zero)| \Big) < \frac{\epsilon}{\log |S| +2} .
\end{equation}

Let $Y^{\text{ord}}$ be an \iid\ process with entropy at most $\epsilon$ and let $\le$ be a total order on $\V$ as guaranteed by Lemma~\ref{lem:total-order} (note that if the graph $G$ contains no edges, then $X$ is already an \iid\ process so that there is nothing to prove).

\subsection{The construction of the finitary coding}

Recall that the $n$-th $\le$-successor of $v$ is denoted by $v+n$ and its $n$-th $\le$-predecessor by $v-n$. In particular, $v \pm n$ are random elements of $\V$ which are determined from $Y^{\text{ord}}$ is a finitary manner. Recall also that the cell process $A$ is a finitary factor of $Y^{\text{cell}}$. It may be helpful from this point onward to think of $A$ and $\le$ as given, and that our goal is to use them, together with the random bits of $Y^{\text{bits}}$, to construct a realization of $X$.

We shall define, for every time $t \ge 0$ and every vertex $u \in \V$, a random variable
\[ L^t_u \in \{0,1\} .\]
We shall define these inductively with the $t=0$ variables given by
\begin{equation}\label{eq:L^0}
L^0_u := \1_{\{u\text{ is a level 1 agent and }|Y^{\text{bits}}_u|>0\}} .
\end{equation}
We think of $L^t_u$ as indicating whether $u$ read a bit at time $t$. In particular, if $L^t_u=1$ for some~$t$ and~$u$, then necessarily $u$ is an agent (of some level).
Since at time 0 an agent looks for an available bit at its own location, \eqref{eq:L^0} says that any level 1 agent reads a bit at time 0 if such a bit is available.

Before giving the main definitions of the construction, we first set up some auxiliary notation and definitions.
As we have already mentioned, our construction has the property that if an agent~$u$ reads a bit at some time $t$, then the bit it read is located at $u+t$, i.e., it is one of the bits of the word $Y^{\text{bits}}_{u+t}$.
In particular, the \emph{total number of bits read from location $v$ by time $t$} is
\[ M^t_v := L^0_v + L^1_{v-1} + \cdots + L^t_{v-t} .\]
The bits at any location $v$ are read sequentially -- the first agent to read a bit at $v$ will read $Y^{\text{bits}}_v(1)$, the second will read $Y^{\text{bits}}_v(2)$ and so on. Precisely,
the \emph{bit read by $u$ at time $t$} is
\[ \hat{W}^t_u := \begin{cases}
 Y^{\text{bits}}_{u+t}(M^t_{u+t}) &\text{if }L^t_u=1\\
 \emptyset &\text{otherwise}
\end{cases}.\]
For this to be well-defined, we must make sure that $u$ does not try to read a non-existent bit -- we mention already here that this does not occur, i.e., our definitions will ensure that
\[ M^t_v \le |Y^{\text{bits}}_v| \qquad\text{for all }v \in \V\text{ and all }t \ge 0 .\]
The \emph{word read by $u$ by time $t$} is then
\[ W^t_u := \hat{W}^0_u \circ \hat{W}^1_u \circ  \dots \circ \hat{W}^t_u ,\]
where $\circ$ denotes concatenation.
That is, $W^t_u$ is the word obtained by concatenating the bits read by $u$ until time $t$ in the order they were read. In particular, $W^t_u$ is a word in $\{0,1\}^*$ of length $|W^t_u|= L^0_u+L^1_u+\dots+L^t_u$.
We emphasize that $(M^t_u,W^t_u)_{u \in \V}$ is well-defined once $(L^i_u,N^i_u)_{u \in \V, 0 \le i \le t}$ is defined, as the former are functions of the latter and of $Y^{\text{bits}}$.

As explained in the proof outline, we use ``simulations'' to obtain samples of distributions from random bits. We first equip ourselves with simulations of all the possible distributions we may require throughout the construction of the finitary coding. The basic distributions we need are those of $X_V$ for a finite set $V \subset \V$. As we aim to obtain a $\Gamma$-equivariant factor, we must take care when dealing with random elements of $S^V$, as these are indexed by subsets of vertices. It would be more proper to view $X_V$ as a random element of $S^{|V|}$ by using the order $\le$. Precisely, we proceed as follows.
Recall the definition of a simulation from Section~\ref{sec:simulation}.
By Theorem~\ref{thm:sim}, for every ordered sequence $v_1,\dots,v_m \in \V$ of distinct vertices, there exists a simulation ${\sf S}_{(v_1,\dots,v_m)}$ of $(X_{v_1},\dots,X_{v_m}) \in S^m$ satisfying that
\[ \E {\sf S}^{\text{time}}_{(v_1,\dots,v_m)}(\omega) \le H(X_V) + 2 .\]
Since the distribution of $X$ is $\Gamma$-invariant, we may suppose that ${\sf S}_{(v_1,\dots,v_m)} = {\sf S}_{(\gamma v_1,\dots,\gamma v_m)}$ for all $\gamma \in \Gamma$ (e.g., by choosing a simulation for a single representative of each orbit, and then setting ${\sf S}_{(v_1,\dots,v_m)}$ to equal the simulation of its representative). Now, for a finite set $V \subset \V$, we let $v_1,\dots,v_m$ be the vertices of $V$, ordered according to $\le$, and set ${\sf S}_V$ to be the simulation $S_{(v_1,\dots,v_m)}$, where, for notational convenience, we interpret ${\sf S}^{\text{out}}_V$ as an element of $S^V$ (indexed by $V$) through the identification ${\sf S}^{\text{out}}_V(\omega)_{v_i} = {\sf S}^{\text{out}}_{(v_1,\dots,v_m)}(\omega)_i$.
We stress that ${\sf S}_V$ implicitly depends on the order $\le$.

As we will also encounter situations in which regions of $X$ have already been sampled, we will also need simulations of the distribution of $X_V$ conditioned on $X_U$ for some finite set $U \subset \V$ which is disjoint from $V$. Thus, for every such $V$ and $U$ and every $\tau \in S^U$, we similarly let ${\sf S}_{V,U,\tau}$ be a simulation of $\Pr(X_V \in \cdot \mid X_U=\tau)$ satisfying that
\begin{equation}\label{eq:sim-time-bound}
\E {\sf S}^{\text{time}}_{V,U,\tau}(\omega) \le H(X_V \mid X_U=\tau) + 2 .
\end{equation}
For ease of notation later on, we allow $V$ and $U$ to intersect and we allow $\tau$ to have any domain containing $U$, by interpreting ${\sf S}_{V,U,\tau}$ as ${\sf S}_{V \setminus U,U,\tau_U}$ in such a case.
We also identify ${\sf S}_V$ with ${\sf S}_{V,\emptyset,\emptyset}$.

Recall that every cell $\cC$ in $A_n$ that is not contained in $A_{n-1}$ has an associated level~$n$ agent, and that this agent is ``responsible'' for generating the output on $\cC \setminus A_{n-1}$.
We denote by $A_n(v)$ the cell of $v$ in $A_n$, where $A_n(v):=\emptyset$ if $v \notin A_n$, by $U_n$ the set of level $n$ agents and, for $v \in A_n \setminus A_{n-1}$, by $U_n(v)$ the level $n$ agent associated to the cell $A_n(v)$.

\smallbreak

With the above notation and definitions, we may now proceed to construct the finitary coding.
Our goal is to define a random field $Z^t=(Z^t_v)_{v \in \V}$, which represents the \emph{output at time $t$}. This output will be a function of $(L^i_u,M^i_u,W^i_u)_{u \in \V, 0 \le i \le t}$ (and of course of the cell process $A$ and the total order $\le$). Once $Z^t$ is defined for some $t$, it will then only remain to inductively define $(L^{t+1}_u)_{u \in \V}$, as this then also defines $(M^i_u,W^i_u)_{u \in \V, 0 \le i \le t+1}$ through the definitions above. This will therefore define $Z^{t+1}$ as well. We will then take a limit as $t \to \infty$ in order to obtain the \emph{output} $Z=(Z_v)_{v \in \V}$, which is the desired realization of $X$.

To facilitate the inductive definition of $(L^{t+1}_u)_{u \in \V}$, we require some more definitions. We now regard $t \ge 0$ as fixed and suppose that $(L^i_u)_{u \in \V, 0 \le i \le t}$, and hence also $(M^i_u,W^i_u)_{u \in \V, 0 \le i \le t}$, are already defined.
We define two notions for a level $n$ agent: that of having reached level $n$ at time $t$, and that of having completed level $n$ of the simulation by time $t$. We define these notions inductively on $n$. We thus begin with level 1 agents.
Given a level $1$ agent $u \in U_1$, we say that
\begin{itemize}
	\item $u$ \emph{reached} level $1$ by time $t$ (always, with no condition).
	\item $u$ \emph{completed} level $1$ by time $t$ if the stopping time ${\sf S}^{\text{time}}_{A_1(u)}$ has been reached on input $W^t_{u,1}$.
\end{itemize}
Once a level 1 agent has completed level 1 of the simulation, the output is known on the corresponding level 1 cell of the agent. That is, if a level 1 agent $u$ completed level 1 by time~$t$, then we will have
\[ Z^t_v = {\sf S}^{\text{out}}_{A_1(u)}(W^t_{u,1})_v \qquad\text{for all }v \in A_1(u) .\]
To be more precise, let us define $Z^{t,1}=(Z^{t,1}_v)_{v \in \V}$ by
\[ Z^{t,1}_v := \begin{cases}
 {\sf S}^{\text{out}}_{A_1(u)}(W^t_u)_v &\text{if }v \in A_1(u)\text{ for some }u \in U_1\text{ and $u$ completed level $1$ by time $t$} \\\emptyset &\text{otherwise}
\end{cases} .\]
We will soon also define $Z^{t,n}=(Z^{t,n}_v)_{v \in \V}$ for $n \ge 2$, with the idea that it represents the known output on all cells of level at most $n$ for which the simulation has completed. In particular, if $Z^{t,n}_v \neq \emptyset$ for some~$v$ and~$n$, then it will be the case that $Z^{t,n+1}_v = Z^{t,n}_v$.
Now fix $n \ge 2$ and suppose that we have defined $Z^{t,1},\dots,Z^{t,n-1}$ and the two notions (reached and completed) for levels less than $n$, in such a way that the above property holds -- namely, if a level $m \in \{1,\dots,n-1\}$ agent~$u$ completed level $m$ by time $t$, then the output is known on $A_m(u)$ at time $t$ in the sense that $Z^{t,n-1}_v = Z^{t,m}_v \neq \emptyset$ for all $v \in A_m(u)$.
Then, for a level $n$ agent $u \in U_n$, we say that
\begin{itemize}
	\item $u$ \emph{reached} level $n$ by time $t$ if every level $n-1$ agent $u' \in U_{n-1} \cap A_n(u)$ has completed level $n-1$ by time $t$.
\end{itemize}
	The idea here is that if $u$ reached level $n$ by time $t$, then the output is known on $A_n(u) \cap A_{n-1}$ at time~$t$, and we may use this information to start generating the output on the remaining part of the cell, namely, on $A_n(u) \setminus A_{n-1}$. That is, we use the simulation ${\sf S}_{V,U,\tau}$ with $V=A_n(u)$, $U=A_n(u) \cap A_{n-1}$ and $\tau = Z^{t,n-1}$. We thus say that
\begin{itemize}
	\item $u$ \emph{completed} level $n$ by time $t$ if it reached level $n$ by time $t$ and the stopping time
	\[ {\sf S}^{\text{time}}_{A_n(u),A_n(u) \cap A_{n-1},Z^{t,n-1}} \]
	has been reached on input $W^t_u$.
\end{itemize}
Putting this together leads to defining $Z^{t,n}=(Z^{t,n}_v)_{v \in \V}$ by
\[ Z^{t,n}_v := \begin{cases}
	{\sf S}^{\text{out}}_{A_n(u),A_n(u) \cap A_{n-1},Z^{t,n-1}}(W^t_u)_v &\substack{\text{if }v \in A_n \setminus A_{n-1}\text{ and }U_n(v)=u\text{ for some }u\\\text{ and $u$ completed level $n$ by time }t}\\
	\emptyset &\text{otherwise}
\end{cases} .\]
The \emph{output $Z^t=(Z^t_v)_{v \in \V}$ at time $t$} is then defined by $Z^t_v := \lim_{n \to \infty} Z^{t,n}_v$.
That is, to determine $Z^t_v$, we first look at the level $n$ at which $v$ enters the cell process, and then consider the level $n$ agent $u$ responsible for generating the output on the cell of $v$ in $A_n$.
If $u$ has indeed completed level $n$ by time $t$, then we read the value of $Z^t_v$ from the output of the corresponding simulation.

Finally, we are ready to define $(L^{t+1}_u)_{u \in \V}$.
As mentioned, these numbers are always zero for non-agents, i.e., we set $L^{t+1}_u:=0$ for $u \notin U_1 \cup U_2 \cup \cdots$. Suppose now that $u \in U_n$ for some $n \ge 1$.
We say that $u$ is \emph{active at time $t+1$} if it has reached, but has not completed, level $n$ by time $t$. Thus, if $u$ is active at time $t+1$, then ideally it would like to read a bit at that time, and indeed it may do so as long as there is an available bit at $u+t+1$ (recall that $u$ may only read a bit from location $u+t+1$ at time $t+1$). This leads us to define
\begin{equation}\label{eq:L^t}
L^{t+1}_u := \1\big(u\text{ is active at time $t+1$ and }M^t_{u+t+1} < |Y^{\text{bits}}_{u+t+1}|\big) .
\end{equation}
This completes the inductive definition of $L^{t+1}_u$ for all $t \ge 0$ and $u \in \V$.

We note that, by construction, once the output at a vertex is determined at some time, it remains unchanged at future times -- that is, if $Z^t_v \neq \emptyset$ for some $v$ and $t$, then $Z^{t+1}_v = Z^t_v$. Denote by
\[ T_v := \min\{t \ge 0 : Z^t_v \neq \emptyset \} \]
the time at which the output at $v$ is first determined.
The \emph{output} $Z=(Z_v)_{v \in \V}$ is then given by
\[ Z_v := \lim_{t \to \infty} Z^t_v = \begin{cases} Z_v^{T_v} &\text{if }T_v<\infty \\ \emptyset &\text{if }T_v=\infty \end{cases} .\] 
This completes the construction of the finitary factor.

We record for later use the following simple property of the construction.

\begin{lemma}\label{lem:time-measurability}
For any $t \ge 0$, we have that $(W^t_u,Z^t_u)_{u \in \V}$ and $(L^{t+1}_u)_{u \in \V}$ are measurable with respect to $Y^{\text{cell}}$, $Y^{\text{ord}}$, $(|Y^{\text{bits}}_v|)_{v \in \V}$ and $(Y^{\text{bits}}_v(i))_{v \in \V, 1 \le i \le M_v^t}$.
\end{lemma}

\begin{proof}
The proof by induction on $t$ is straightforward from the definitions.
\end{proof}

\subsection{Concluding Theorem~\ref{thm:main}}

To conclude the proof of Theorem~\ref{thm:main}, we must establish two properties of the above construction: that the output it produces has the desired distribution, and that the output can be determined from $(Y^{\text{bits}},Y^{\text{cell}},Y^{\text{ord}})$ in a finitary manner.
The former is stated in the following proposition whose proof is postponed to Section~\ref{sec:correct-distribution} below.

\begin{prop}\label{prop:correct-distribution}
The output $Z$ has the same distribution as $X$.
\end{prop}

\begin{proof}[Proof of Theorem~\ref{thm:main}]
The random field $Z$ is clearly a deterministic and $\Gamma$-equivariant function $\varphi$ of $(Y^{\text{bits}},Y^{\text{cell}},Y^{\text{ord}})$. Thus, in light of Proposition~\ref{prop:correct-distribution}, we must only show that $\varphi$ is finitary. Since $T_v$ is almost surely finite (as $Z_v \neq \emptyset$ almost surely by Proposition~\ref{prop:correct-distribution}), it suffices to show that $Z^t$ is finitary for every $t \ge 0$. This follows rather easily from the construction. To see this, we explain how to determine the value of $Z^t_v$ in a finitary manner.

We begin by finding the level $n$ in which $v$ enters the cell process, i.e., $v \in A_n \setminus A_{n-1}$, and then finding the cell $A_n(v)$ of $v$ in $A_n$. Since $A$ is a finitary factor of $Y^{\text{cell}}$, this may be done in a finitary manner.
Next, we find the level $n$ agent $U_n(v)$ associated to the cell $A_n(v)$. Since this is just the $\le$-minimal element in $A_n(v) \setminus A_{n-1}$, and since the order $\le$ is a finitary factor of $Y^{\text{ord}}$, this may also be done in a finitary manner.

Let us suppose by induction that all steps of the construction up to time $t-1$ are finitary.
Thus, recalling the definition of active, we see that, for any vertex $w$, we may determine in a finitary manner whether $w$ is active at time $t$. Since successors/predecessors in $\le$ may be found in a finitary manner from $Y^{\text{ord}}$, it then also follows that $L^t_w$ may be determined in a finitary manner. Using again that successors/predecessors may be found in a finitary manner, we conclude that $W^t_w$ may be found in a finitary manner.

We would now like to check whether $u$ completed level $n$ by time $t$, and if so, find the output value.
To check this, we start at level 1 and work our way up to level $n$.
Thus, we first find all level $1$ agents which are contained in $A_n(v)$ (since the cell process and total order are finitary, this can be done in a finitary manner). Next, for each such agent $u$, we check whether $u$ completed level~1 by time $t$. Recall that the simulation ${\sf S}_{A_1(u)}$ depends on the cell $A_1(u)$ and on the order induced by $\le$ on $A_1(u)$. Since the input word $W^t_u$, the cell process and the order are finitary, we see that we may determine whether $u$ completed level 1 by time $t$ in a finitary manner, and if so, also determine the output $Z^{t,1}_w$ for all $w \in A_1(u)$ in a finitary manner.

We now proceed to the next levels. Consider some level $2 \le m \le n$. We again begin by finding all level $m$ agents which are contained in $A_n(v)$. For each such agent $u$, we check whether $u$ reached level $m$ by time $t$. For this we must check whether the level $m-1$ agents in $A_m(u)$ completed level $m-1$ by time $t$, which, by induction, may be done in a finitary manner. If $u$ reached level $m$ by time $t$, we then check whether $u$ completed level $m$ by time $t$. Similarly to before, the simulation ${\sf S}_{A_m(u),A_m(u) \cap A_{m-1},Z^{t,m-1}}$ depends on $A_m(u)$, $A_m(u) \setminus A_{m-1}$, the order induced by $\le$ on $A_m(u)$, and on $(Z^{t,m-1}_w)_{w \in A_m(u) \cap A_{m-1}}$. Since the input word $W^t_u$, the cell process and the order are finitary, we see that we may determine whether $u$ completed level $m$ by time $t$ in a finitary manner, and if so, also determine the output $Z^{t,m}_w$ for all $w \in A_m(u)$ in a finitary manner.

Continuing up to level $m=n$ yields that $Z^{t,n}_v$ may be determined in a finitary manner.
Since $n$ is the level in which $v$ enters the cell process, we have by definition that $Z^t_v = Z^{t,n}_v$.
Thus, $Z^t_v$ may be determined in a finitary manner, as required.
\end{proof}

\subsection{The output has the correct distribution}
\label{sec:correct-distribution}

In this section, we prove Proposition~\ref{prop:correct-distribution}.
The proof is split up into several steps. The first step is the following lemma which formalizes the intuition that the simulations used in the construction are `fed' independent unbiased bits.

\begin{lemma}\label{lem:distrubition-of-words}
Let $\omega \in \{0,1\}^\N$ consist of a sequence of independent unbiased bits. Let $(\omega_u)_{u \in \V}$ be a collection of \iid\ copies of $\omega$, independent of $(Y^{\text{bits}},Y^{\text{cell}},Y^{\text{ord}})$.
Then, for any $t \ge 0$, conditioned on $(Y^{\text{cell}},Y^{\text{ord}})$, the collection $(W^t_u \circ \omega_u)_{u \in \V}$ has the same distribution as $(\omega_u)_{u \in \V}$.
\end{lemma}
\begin{proof}
We prove the statement by induction on $t$, taking $t=-1$ as a trivial base case (where $W^{-1}_u:=\emptyset$ for all $u \in \V$). Suppose now that we know it for some $t \ge -1$ and let us show it for $t+1$.
Recall that $W^{t+1}_u = W^t_u \circ \hat{W}^{t+1}_u$.
Thus, we need to show that, conditioned on $(Y^{\text{cell}},Y^{\text{ord}})$, the collection $(W^t_u \circ \hat{W}^{t+1}_u \circ \omega_u)_{u \in \V}$ has the same distribution as $(\omega_u)_{u \in \V}$.
To this end, it suffices to show that, conditioned on $(Y^{\text{cell}},Y^{\text{ord}})$, the collections $(W^t_u)_{u \in \V}$ and $(\hat{W}^{t+1}_u \circ \omega_u)_{u \in \V}$ are independent and that the conditional distribution of the latter is that of $(\omega_u)_{u \in \V}$. Indeed, the induction hypothesis will then yield the desired result.

We may restate our goal as showing that, conditioned on $(Y^{\text{cell}},Y^{\text{ord}})$ and $(W^t_u)_{u \in \V}$, the collection $(\hat{W}^{t+1}_u \circ \omega_u)_{u \in \V}$ has the same distribution as $(\omega_u)_{u \in \V}$.
Let $\cF$ be the $\sigma$-algebra generated by $Y^{\text{cell}}$, $Y^{\text{ord}}$, $(|Y^{\text{bits}}_v|)_{v \in \V}$ and $(Y^{\text{bits}}_v(i))_{v \in \V, 1 \le i \le M^t_v}$.
By Lemma~\ref{lem:time-measurability},
\[ (W^t_u)_{u \in \V} \qquad\text{and}\qquad Q := \big\{ u : \hat{W}^{t+1}_u \neq \emptyset \big\} = \big\{ u : L^{t+1}_u=1 \big\} \]
are $\cF$-measurable. Since $(\omega_u)_{u \in \V}$ is independent of $(Y^{\text{bits}},Y^{\text{cell}},Y^{\text{ord}})$ and hence also of $\cF$, it suffices to show that, conditioned on $\cF$, the random variables $(\hat{W}^{t+1}_u)_{u \in Q}$ are independent unbiased bits.

Note that $Q$ is the set of vertices (agents) that read a bit at time $t+1$, that
\[ Q':= \{ v : M^{t+1}_v>M^t_v \} \]
is the set of vertices from which a bit was read at time $t+1$, and that $u \mapsto u+t+1$ defines a $\cF$-measurable bijection from $Q$ to $Q'$. Recall also that $\hat{W}^{t+1}_u=Y^{\text{bits}}_{u+t+1}(M^{t+1}_{u+t+1})$ for $u \in Q$.
Thus, it suffices to show that, conditioned on $\cF$, the random variables $(Y^{\text{bits}}_v(M^{t+1}_v))_{v\in Q'}$ are independent unbiased bits.

Indeed, since $Q'$ and $(M^{t+1}_v)_{v \in Q'}$ are $\cF$-measurable by Lemma~\ref{lem:time-measurability}, since $M^{t+1}_v>M^t_v$ for all $v \in Q'$, and since $Y^{\text{bits}}$ is an \iid\ process that is independent of $(Y^{\text{cell}},Y^{\text{ord}})$, we see that the random variables $(Y^{\text{bits}}_v(M^{t+1}_v))_{v \in Q'}$ are conditionally independent given $\cF$, and that, for any $v \in Q'$, the conditional distribution of $Y^{\text{bits}}_v(M^{t+1}_v)$ is the same as the distribution of $Y^{\text{bits}}_v(M^{t+1}_v)$ given $|Y^{\text{bits}}_v|$. Since $Y^{\text{bits}}_v$ is a random number of random bits, the latter is the distribution of an unbiased bit, and the proof is complete.
\end{proof}

We will use the above lemma for fixed $t$ and then let $t$ tend to infinity. In doing so, we will encounter the limiting word $W^\infty_u := \lim_{t \to \infty} W^t_u$. Since $W^{t+1}_u$ extends $W^t_u$, this limit is well-defined and is a word in $\{0,1\}^*$ or $\{0,1\}^\N$ (we will see that it is in fact a finite word almost surely).

The next step towards proving Proposition~\ref{prop:correct-distribution} is to show that the output at every vertex $v$ is eventually determined, i.e., that $Z_v \neq \emptyset$ (equivalently, $T_v < \infty$) almost surely. For this, we first show that every vertex is eventually inactive.

\begin{lemma}\label{lem:eventually-inactive}
Every vertex is almost surely eventually inactive. That is, for any $u \in \V$, there almost surely exists a finite $t_0$ such that $u$ is not active at any time $t \ge t_0$.
\end{lemma}
\begin{proof}
Define
\[ \psi(u,v) := \1_{\{v=u+t\text{ and }L^t_u=1\text{ for some }t \ge 0\}} .\]
Note that $\psi(u,v)$ indicates whether $u$ read a bit located at $v$. Since an agent may read at most one bit from any location, $\psi(u,v)$ also represents the \emph{number} of bits read by $u$ from location~$v$.
Thus, recalling that $|W^t_u| = L^0_u+L^1_u + \dots + L^t_u$, we have
\[ \sum_v \psi(u,v) = \sum_{t=0}^\infty L^t_u = |W^\infty_u| \qquad\text{and}\qquad \sum_u \psi(u,v) = \sum_{t=0}^\infty L^t_{v-t} = \lim_{t\to\infty} M^t_v =: M^\infty_v .\]
The left-hand side describes the number of bits read \emph{by} a given site $u$, while the right-hand side describes the number of bits read \emph{from} a given site $v$.
The mass-transport principle~\eqref{eq:mass-transport} tells us that these quantities are the same in expectation:
\begin{equation}\label{eq:L^t-vs-M^infty}
\E |W^\infty_u| = \E M^\infty_v .
\end{equation}

Let $E_u$ be the event that $u$ is active at infinitely many times $t$. We wish to show that $\Pr(E_u)=0$. Note that, by~\eqref{eq:L^t}, the event $E_u$ is contained in the event that for all but finitely many $t \ge 0$, all bits at location $u+t$ have been read by time $t$, i.e.,
\begin{align*}
 E_u &\subset \big\{M^t_{u+t} \ge |Y^{\text{bits}}_{u+t}|\text{ for all sufficiently large }t \big\} \\
 &= \big\{M^\infty_{u+t} = |Y^{\text{bits}}_{u+t}|\text{ for all sufficiently large }t \big\} ,
\end{align*}
where the equality follows from the fact that $M^t_v \le M^\infty_v \le |Y^{\text{bits}}_v|$ for all $v \in \V$ and $t \ge 0$.
Suppose now that $\Pr(E_u)>0$. Then by ergodicity, almost surely, $E_w$ occurs for some $w \in \V$, and in particular, there almost surely exists $w \in \V$ such that $M^\infty_{w+i}=|Y^{\text{bits}}_{w+i}|$ for all $i \ge 0$. Since $\{w+i\}_{i \in \Z}=\V$ almost surely, it follows by $\Gamma$-invariance that $M^\infty_v=|Y^{\text{bits}}_v|$ for all $v \in \V$ almost surely.
Thus, by~\eqref{eq:L^t-vs-M^infty},
\begin{equation}\label{eq:L^t-vs-Y^bits}
\E |W^\infty_u| = \E |Y^{\text{bits}}_v| .
\end{equation}
That is, the expected number of bits read by each site is precisely the expected number of available bits per site.
It remains to show that this is impossible.

Define
\[ \phi(u,v) := \begin{cases} \frac{|W^\infty_u|}{|A_n(u) \setminus A_{n-1}|} &\text{if }v \in A_n \setminus A_{n-1}\text{ and }U_n(v)=u \\0&\text{otherwise} \end{cases} .\]
Recall that $U_n(v)$ is the level $n$ agent associated to the cell $A_n(v)$. Since a level $n$ agent $u$ is responsible for simulating the output on $A_n(u) \setminus A_{n-1}$ and does so via the input word $W^\infty_u$, we may think of $\phi(u,v)$ as follows: every level $n$ agent $u$ equally divides a total `cost' of $|W^\infty_u|$ among the vertices it `serviced'. Observe that
\[ \sum_v \phi(u,v) = |W^\infty_u| \qquad\text{and}\qquad \sum_u \phi(u,v) = \frac{|W^\infty_{U_{N_v}(v)}|}{|A_{N_v}(v) \setminus A_{N_v-1}|} ,\]
where $N_v$ is the level at which $v$ entered the cell process, i.e., $v \in A_{N_v} \setminus A_{N_v-1}$, and where we used that $A_n(v)=A_n(u)$ whenever $U_n(v)=u$.
Thus, by~\eqref{eq:L^t-vs-Y^bits} and the mass-transport principle~\eqref{eq:mass-transport},
\begin{equation}\label{eq:Y^bits-vs-W^infty}
\E|Y^{\text{bits}}_v| = \E\left[\frac{|W^\infty_{U_{N_v}(v)}|}{|A_{N_v}(v) \setminus A_{N_v-1}|}\right] .
\end{equation}
This relates the expected number of available bits per site to the length of the input words used by the simulations. We would like to reach a contradiction to the fact that there are many available bits and that the simulation is efficient.

Suppose that $u$ is a level $n$ agent.
It is straightforward from the definitions that the stopping time ${\sf S}^{\text{time}}_{A_n(u),A_n(u) \cap A_{n-1},Z^{t,n-1}}$ is not reached on any prefix of $W^t_u$ that is not $W^t_u$ itself (it may or may not be reached on the entire word $W^t_u$). It therefore follows from Lemma~\ref{lem:distrubition-of-words} that, conditioned on $(Y^{\text{cell}}, Y^{\text{ord}})$,
\[ |W^t_u| \text{ is stochastically dominated by } {\sf S}^{\text{time}}_{A_n(u),A_n(u) \cap A_{n-1},Z^{t,n-1}}(\omega) \cdot \1(\text{$u$ reached level $n$ by time $t$}), \]
where $\omega \in \{0,1\}^\N$ consists of a sequence of independent unbiased bits, independent of $Y$.
Note that, if $u$ reached level $n$ by time $t$, then $Z^{t,n-1}$ coincides with $Z$ on $A_n(u) \cap A_{n-1}$.
Thus, taking expectations and $t \to \infty$, we obtain that
\[ \E\left[ |W^\infty_u| \mid Y^{\text{cell}}, Y^{\text{ord}} \right] \le \E\left[ {\sf S}^{\text{time}}_{A_n(u),A_n(u) \cap A_{n-1},Z}(\omega) \mid Y^{\text{cell}}, Y^{\text{ord}} \right] \cdot \1(\text{$u$ eventually reached level $n$}) .\]
Hence, by~\eqref{eq:sim-time-bound} and~\eqref{eq:entropy-trivial-bound}, on the event that $u \in U_n$, we have
\[ \E\left[ |W_u^\infty| \mid Y^{\text{cell}}, Y^{\text{ord}} \right] \le 2 + \begin{cases}
 H_X(A_1(u)) &\text{if }n=1\\
 |A_n(u) \setminus A_{n-1}| \cdot \log |S| &\text{if }n \ge 2
\end{cases} ,\]
where we denote $H_X(V) := H(X_V)$ for a finite set $V \subset \V$.
Therefore, by~\eqref{eq:entropy-via-boundary} and the choice of $\delta$,
\[ \E\left[\frac{|W^\infty_{U_{N_v}(v)}|}{|A_{N_v}(v) \setminus A_{N_v-1}|} \mid Y^{\text{cell}}, Y^{\text{ord}} \right] \le (h(X)+ \epsilon + 2\delta) \cdot \1_E + (\log|S| + 2) \cdot \1_{E^c} ,\]
where $E$ is the event that $N_v=1$ and $|\partial A_1(v)| \le \delta |A_1(v)|$.
Thus, by~\eqref{eq:Y^bits-vs-W^infty},
\[ \E|Y^{\text{bits}}_v| \le (h(X)+ \epsilon + 2\delta) \cdot \Pr(E) + (\log|S|+2) \cdot \Pr(E^c) .\]
Using~\eqref{eq:delta} and~\eqref{eq:small-cluster-boundary}, we see that $\E|Y^{\text{bits}}_v| < h(X)+3\epsilon$, which contradicts~\eqref{eq:entropy-choice}. We therefore conclude that $\Pr(E_u)=0$ as required.
\end{proof}

We are now ready to show that the output at every vertex is eventually determined.

\begin{lemma}\label{prop:ouput-known-in-finite-time}
For any $v \in \V$, we have that $T_v < \infty$ almost surely.
\end{lemma}
\begin{proof}
Since $A_n$ almost surely increases to $\V$, it suffices to show that $\Pr(Z_v=\emptyset\text{ and }v \in A_n)=0$ for all $n \ge 1$.
We prove this by induction on $n$, taking $n=0$ as a trivial base case by setting $A_0 := \emptyset$.
Let $n \ge 1$ and suppose that $\Pr(Z_v=\emptyset\text{ and }v \in A_{n-1})=0$. By $\Gamma$-invariance, we actually have that $\Pr(Z_w=\emptyset\text{ for some }w \in A_{n-1})=0$.
We may thus assume that $Z_w \neq \emptyset$ for all $w \in A_{n-1}$.
Suppose now that $Z_v=\emptyset$ and $v \in A_n$.
Let $u$ be the level $n$ agent $U_n(v)$ associated to the cell $A_n(v)$. Observe that, by the definition of $Z_v$ and $Z^{t,n}_v$, we have that, for all $t \ge 0$, $u$ did not complete level~$n$ by time $t$. On the other hand, since $Z_w \neq \emptyset$ for all $w \in A_{n-1}$, there exists a finite $t_0 \ge 0$ such that $u$ has reached level~$n$ by time $t_0$. It follows that $u$ is active at time $t$ for every $t > t_0$. By Lemma~\ref{lem:eventually-inactive}, almost surely, no vertex is active at infinitely many times, thus completing the proof that $\Pr(Z_v=\emptyset\text{ and }v \in A_n)=0$.
\end{proof}

Now that we have established that the output at every vertex is eventually determined, it remains to show that the distribution of the output is the correct one, namely, that of $X$. The following immediately implies Proposition~\ref{prop:correct-distribution}.

\begin{prop}\label{prop:correct-cond-distribution-on-cell}
Conditioned on $(Y^{\text{cell}},Y^{\text{ord}})$, $Z$ almost surely has the same distribution as $X$, where we regard $X$ as independent of $(Y^{\text{cell}},Y^{\text{ord}})$.
\end{prop}

\begin{proof}
Throughout the proof, we regard $X$ as independent of $Y^{\text{cell}}$ and $Y^{\text{ord}}$. We also condition on $(Y^{\text{cell}},Y^{\text{ord}})$ throughout the entire proof, without explicitly mentioning this. In particular, \emph{any} statement about distributions or independence should be understood as conditional on $(Y^{\text{cell}},Y^{\text{ord}})$.

Since every finite subset of~$\V$ is almost surely contained in some cell of the cell process, it suffices to show that, for any $n \ge 1$ and any cell $\cC$ of $A_n$, $Z_\cC$ has the same distribution as $X_\cC$.
We prove this by induction on $n$, taking $n=0$ as a trivial  base case (where $A_0:=\emptyset$).

Suppose now that $n \ge 1$. Let $\cC$ be a cell of $A_n$ and denote $\cC' := \cC \cap A_{n-1}$.
We will show that
\begin{equation}\label{eq:distr-on-previous-cell}
Z_{\cC'} \eqd X_{\cC'}
\end{equation}
and
\begin{equation}\label{eq:distr-cond-on-previous-cell}
\Pr(Z_{\cC \setminus \cC'} \in \cdot \mid Z_{\cC'}=\tau)=\Pr(X_{\cC \setminus \cC'} \in \cdot \mid X_{\cC'}=\tau) \qquad\text{for any feasible }\tau \in S^{\cC'} .
\end{equation}
By feasible $\tau$, we mean that $\Pr(X_{\cC'}=\tau)>0$.
The desired equality in distribution $Z_{\cC} \eqd X_{\cC}$ follows immediately from~\eqref{eq:distr-on-previous-cell} and~\eqref{eq:distr-cond-on-previous-cell}.
Both parts require some type of independence, which we now establish.

Let $\omega \in \{0,1\}^\N$ consist of a sequence of independent unbiased bits. Let $(\omega_u)_{u \in \V}$ be a collection of \iid\ copies of $\omega$, independent of $Y^{\text{bits}}$.
By Lemma~\ref{lem:distrubition-of-words}, for any $t \ge 0$, $(W^t_u \circ \omega_u)_u$ has the same distribution as $(\omega_u)_u$.
Taking the limit as $t \to \infty$, we see that $(W^\infty_u \circ \omega_u)_u$ also has the same distribution as $(\omega_u)_u$.

Observe that, by construction, if $C$ is some cell of the cell process, then $Z^t_C$ is a function of $(W^t_u)_{u \in C}$. Taking the limit as $t \to \infty$, it follows that $Z_C$ is a function of $(W^\infty_u)_{u \in C}$. It also follows from the definition of $Z$ and the fact that $T_v < \infty$ for all $v$, that $Z$ is unchanged by concatenating any word to any $W^\infty_u$. In particular, $Z_C$ is also a function of $(W^\infty_u \circ \omega_u)_{u \in C}$.

We now show~\eqref{eq:distr-on-previous-cell}.
To this end, let $\cC_1,\dots,\cC_m$ be the cells in $A_{n-1}$ that are contained in $\cC$, so that $\cC' = \cC_1 \cup \dots \cup \cC_m$.
By the induction hypothesis, $Z_{\cC_j} \eqd X_{\cC_j}$ for every $1 \le j \le m$.
Since $A$ is a cell process, we have that $\dist(\cC_j,\cC_{j'})>1$ for $1 \le j < j' \le m$. Hence, using that $X$ is 1-dependent, we see that $\{X_{\cC_j}\}_{1 \le j \le m}$ are independent. Thus, it remains to show that $\{Z_{\cC_j}\}_{1 \le j \le m}$ are also independent. Since $Z_{\cC_j}$ is a function of $(W^\infty_u\circ \omega_u)_{u \in \cC_j}$, this follows from the fact that $\{W^\infty_u \circ \omega_u\}_{u \in \V}$ are independent.

To complete the proof, it remains to show~\eqref{eq:distr-cond-on-previous-cell}.
Let $w$ be the agent associated to~$\cC$ and recall that $w \in \cC \setminus \cC'$ and that $\cC=A_n(w)$. Note that
\[ Z_{\cC \setminus \cC'} = {\sf S}^{\text{out}}_{\cC,\cC',Z}(W^\infty_w) = {\sf S}^{\text{out}}_{\cC,\cC',Z}(W^\infty_w \circ \omega_w) .\]
Recall that ${\sf S}_{\cC,\cC',Z}$ is shorthand for ${\sf S}_{\cC \setminus \cC',\cC',Z_{\cC'}}$.
Since $W^\infty_w \circ \omega_w$ is independent of $(W^\infty_u \circ \omega_u)_{u \neq w}$, and hence also of $Z_{\cC'}$, we conclude that the conditional distribution of $Z_{\cC \setminus \cC'}$ given that $Z_{\cC'}=\tau$ is equal to the distribution of ${\sf S}^{\text{out}}_{\cC \setminus \cC',\cC',\tau}(W^\infty_w \circ \omega_w)$. Thus, using that $W^\infty_w \circ \omega_w$ has the same distribution as $\omega$, we see that the distribution in question is that of ${\sf S}^{\text{out}}_{\cC \setminus \cC',\cC',\tau}(\omega)$, which is by definition $\Pr(X_{\cC \setminus \cC'} \in \cdot \mid X_{\cC'}=\tau)$, as required.
\end{proof}

\section{Remarks and open problems}
\label{sec:conclusion}

\begin{remark}
\label{rem:other-thm-proof}
We have given the details of the proof of Theorem~\ref{thm:main}. Theorem~\ref{thm:main-no-entropy} follows the same lines of proof, with minor modifications, all of which are in fact simplifications.

To obtain a proof of Theorem~\ref{thm:main-no-entropy} with the \emph{least} modifications to the existing proof, we may replace the random total order constructed in Lemma~\ref{lem:finiteordering} with the total order induced by an \iid\ process consisting of uniform $[0,1]$ random variables. Using this order in the proof of Lemma~\ref{lem:total-order} yields a random total order with the same properties as in Lemma~\ref{lem:total-order} (except for the bound on the entropy of the \iid\ process).
When $X$ is finite-valued, the proof then goes through with no further modifications. Otherwise, we let $Y^{\text{bits}}_v$ consist of infinitely many independent random bits, and then the proof goes through after an additional modification to the proof of Lemma~\ref{lem:eventually-inactive} (which relied on the fact that the entropy of $X$ is finite to deduce a bound on the expected number of bits used by the simulation; instead we only rely on the fact that, almost surely, the simulation uses only finitely many bits; see Theorem~\ref{thm:sim}).

It is instructive to note that a shorter and conceptually simpler proof exists when one does not need to worry about the entropy of the \iid\ process. This is essentially what is described in ``constructing a finitary coding'' in Section~\ref{sec:outline}. One way to implement the described coding would be to simply replace $u \pm t$ with $u$ everywhere in the construction in Section~\ref{sec:coding}. That is, instead of having an agent $u$ try to read an unused bit from location $u+t$ at time $t$, it always reads bits located at $u$. Since we may place an infinite sequence of bits at every vertex, it will never run out of available bits. In this way, there is no ``moving around'' of bits from one location to another. This could be made conceptually even simpler if instead of using simulations from random bits to obtain samples of distributions as they are needed, from the start, each $Y^{\text{bits}}_v$ is a collection $(W_{V,U,\tau})_{V,U,\tau}$ of independent random variables having distribution $\Pr(X_V \in \cdot \mid X_U = \tau)$ for all finite $U,V \subset \V$ and $\tau \in S^U$. Either way, a nice feature of this construction is that the coding radius depends only on the cell process $A$ constructed in Section~\ref{sec:cell-process}. Namely, the coding radius for determining $X_\zero$ is at most the maximum of $\min\{ r \ge 0 : A(\zero) \subset 
\Lambda_r(\zero \}\}$, where $A(\zero)$ is the cell of $\zero$ in $A_{\min \{ n \ge 1 : \zero \in A_n \}}$, the coding radius for determining the cell $A(\zero)$ and the coding radius for determining the cell process on $A(\zero)$, i.e., $(A_n \cap A(\zero))_{n \ge 1}$. We elaborate on this in the next remark.
\end{remark}

\begin{remark}
\label{rem:coding-radius}
Our main theorems give no information about the coding radius beyond its almost-sure finiteness. However, some information about the coding radius may be extracted from the proof given here. Specifically, Theorem~\ref{thm:main-no-entropy} may be enhanced to give a universal bound on the tail of the coding radius for any fixed graph and finite-dependence range. More precisely, for any transitive amenable graph $G$ and any integer $k \ge 1$, there exists a sequence $(c_n)_{n=1}^\infty$ tending to zero such that any $k$-dependent invariant random field $X$ on $G$ is a finitary factor of an \iid\ process with a coding radius $R$ satisfying that $\Pr(R \ge n) \le c_n$ for all~$n$. The sequence $(c_n)$ depends only on the graph $G$ and on the parameter $k$, and not on the group $\Gamma$ nor on the random field~$X$. Indeed, the sequence $(c_n)$ is governed by the properties of the cell process (see the last part of the previous remark). In particular, for many concrete choices of $G$ (and $k$), an explicit sequence $(c_n)$ may be found.

To illustrate this in a simple setting, let us show that for $G=\Z$ and $k=1$, one may take $c_n = 8/n$. In this case, instead of using the construction given in Section~\ref{sec:cell-process}, it is simpler to consider the finitary cell process $A$ given by $A_n := B_1 \cup \dots \cup B_n$, where $(B_n)_{n \ge 1}$ are independent random subsets of $\Z$, each being an independent Bernoulli percolation with parameter $1/2$. Then the level $N := \min\{ n : 0 \in A_n \}$ at which $0$ enters the cell process is a geometric random variable with parameter $1/2$, conditioned on which, the lengths $L^{\pm} := \min\{ m \ge 1 : \pm m \notin A_N \}$ of the cell of $0$ in $A_N$ in the positive/negative directions are (independent) geometric random variables with parameter $2^{-N}$, and the coding radius $R$ for determining $X_0$ is bounded by $\max \{L^+,L^-\}$. Thus,
\[ \Pr(R>r) \le 2 \cdot \E \big[ (1-2^{-N})^r \big] \le 2 \sum_{n=1}^\infty 2^{-n} e^{-r2^{-n}} \le \frac{4}{r} \sum_{m=-\infty}^\infty 2^m e^{-2^m} \le \frac{8}{r} ,\]
where we used the substitution $n=\lfloor \log_2 r \rfloor - m$.

We remark that if one would like to simultaneously control also the entropy of the \iid\ process (as in Theorem~\ref{thm:main}), then it is plausible that this can be done by allowing $(c_n)$ to depend on the entropy gap $\epsilon$ (and perhaps on $|S|$), but we did not pursue this.
\end{remark}

\begin{remark}\label{rem:spere-condition}
We do not know whether condition~\eqref{eq:sphere-condition} is necessary as stated in Theorem~\ref{thm:main}, however, as we now explain, some condition of this form is needed (i.e., the condition cannot be completely dropped).
Let $G$ be an infinite transitive graph on vertex set $\V$ and let $H$ be a finite transitive graph on $m \ge 2$ vertices. Let $G'$ be the graph obtained by replacing each vertex of $G$ with a copy of $H$ -- that is, the vertex set of $G'$ is $\V \times \{1,\dots,m\}$, and $(u,i)$ and $(v,j)$ are adjacent in $G'$ if and only if $u$ and $v$ are adjacent in $G$, or $u=v$ and $i$ and $j$ are adjacent in $H$.
Any graph $G'$ obtained in this manner is transitive, but fails to satisfy~\eqref{eq:sphere-condition}. Indeed, the balls of radius 2 (or even 1 when $H$ is a complete graph) around $(v,i)$ and $(v,j)$ coincide.
A simple case to have in mind is when $G=\Z$ and $H$ consists of an edge on two vertices, so that the vertices of $G'$ are $\Z \times \{0,1\}$ and there is an edge between $(u,i)$ and $(v,j)$ if and only if $|u-v| \le 1$.

Let $G'$ be any graph as above and let $(W_v)_{v \in \V}$ be independent uniform random variables on $\{1,\dots,m\}$.
Consider the random field $X$ on $G'$ defined by $X_{(v,i)}:=\1_{\{W_v=i\}}$. It is clear that $X$ is 2-dependent and $\text{Aut}(G')$-invariant. We claim that $X$ is not a $\text{Aut}(G')$-factor of any \iid\ process $Y$ on $G'$ whose single-site distribution has at least one atom (in particular, $Y$ cannot have finite entropy). Indeed, for any such process $Y$, the event $Y_{(v,1)}=\dots=Y_{(v,m)}$ has positive probability, and on this event there is no $\text{Aut}(G')$-equivariant way to distinguish between $(v,1),\dots,(v,m)$. That is, any $\text{Aut}(G')$-equivariant function $\varphi \colon T^{\V \times \{1,\dots,m\}} \to \{0,1\}^{\V \times \{1,\dots,m\}}$ must satisfy $\varphi(y)_{(v,1)}=\dots=\varphi(y)_{(v,m)}$ whenever $y \in T^{\V \times \{1,\dots,m\}}$ is such that $y_{(v,1)}=\dots=y_{(v,m)}$. In particular, the event $\varphi(Y)_{(v,1)}=\dots=\varphi(Y)_{(v,m)}$ has positive probability, and hence, $\varphi(Y)$ cannot have the same distribution as $X$.

We remark that there are transitive subgroups $\Gamma$ of $\text{Aut}(G')$ for which the above obstruction does not exist.
For example, let $\Gamma$ be the subgroup of $\text{Aut}(G')$ generated by $\text{Aut}(G)$ and $\text{Aut}(H)$, both of which are naturally embedded in $\text{Aut}(G')$.
Simple modifications to the proofs of Lemma~\ref{lem:finiteordering} and Lemma~\ref{lem:total-order} yield a $\Gamma$-invariant random total order with the desired properties. The rest of the proof then goes through unchanged showing that our main result holds in this case: any finitely dependent $\Gamma$-invariant process on $G'$ is a finitary $\Gamma$-factor of an \iid\ process with slightly larger entropy.
We believe that our main result holds in many similar situations, where condition~\eqref{eq:sphere-condition} is replaced by a suitable condition on $\Gamma$. We did not pursue this direction.
\end{remark}

\medbreak\noindent{\bf Open problems.}
\begin{enumerate}
	\item One may wonder about the situation on non-amenable graphs such as a regular tree (of degree at least three). Namely, is every automorphism-invariant finitely dependent process on a tree a finitary factor of an \iid\ process? In fact, even the more fundamental question of whether such a process is a factor of \iid\ (without the finitary condition) is still open; see~\cite[Question~2.2]{lyons2017factors}. The same questions may be asked on any transitive non-amenable graph.
	\item Does there exist a stationary finitely dependent process on $\Z$ (or, more generally, on some transitive amenable graph) that cannot be expressed as a finitary factor of an \iid\ process with finite expected coding radius? As mentioned, the 1-dependent 4-coloring and 2-dependent 3-coloring of~\cite{holroyd2016finitely} are believed to be examples of such processes, but this is still unproved.
	\item A finitary isomorphism is a finitary factor that is invertible and whose inverse is also finitary.
	Somorodinky~\cite{smorodinsky1992finitary} showed that every stationary finitely dependent process on $\Z$ is finitarily isomorphic to an \iid\ process. Is this true in higher dimensions? Namely, is every stationary finitely dependent process on $\Z^d$ finitarily isomorphic to an \iid\ process?
\end{enumerate}

\bibliographystyle{amsplain}
\bibliography{library}

\begin{bibdiv}
\begin{biblist}

\bib{aaronson1992structure}{article}{
      author={Aaronson, Jon},
      author={Gilat, David},
      author={Keane, Michael},
       title={On the structure of 1-dependent {M}arkov chains},
        date={1992},
     journal={Journal of Theoretical Probability},
      volume={5},
      number={3},
       pages={545\ndash 561},
}

\bib{aaronson1989algebraic}{article}{
      author={Aaronson, Jon},
      author={Gilat, David},
      author={Keane, Michael},
      author={de~Valk, Vincent},
       title={An algebraic construction of a class of one-dependent processes},
        date={1989},
     journal={The annals of probability},
       pages={128\ndash 143},
}

\bib{benjamini1999group}{article}{
      author={Benjamini, Itai},
      author={Lyons, Russell},
      author={Peres, Yuval},
      author={Schramm, Oded},
       title={Group-invariant percolation on graphs},
        date={1999},
     journal={Geometric \& Functional Analysis GAFA},
      volume={9},
      number={1},
       pages={29\ndash 66},
}

\bib{van1999existence}{article}{
      author={Berg, Jacob van~den},
      author={Steif, Jeffrey~E},
       title={On the existence and nonexistence of finitary codings for a class
  of random fields},
        date={1999},
     journal={Annals of probability},
       pages={1501\ndash 1522},
}

\bib{burton19931}{article}{
      author={Burton, Robert~M},
      author={Goulet, Marc},
      author={Meester, Ronald},
      author={others},
       title={On 1-dependent processes and $ k $-block factors},
        date={1993},
     journal={The Annals of Probability},
      volume={21},
      number={4},
       pages={2157\ndash 2168},
}

\bib{haggstrom1997infinite}{article}{
      author={H{\"a}ggstr{\"o}m, Olle},
       title={Infinite clusters in dependent automorphism invariant percolation
  on trees},
        date={1997},
     journal={The Annals of Probability},
       pages={1423\ndash 1436},
}

\bib{harel2018finitary}{article}{
      author={Harel, Matan},
      author={Spinka, Yinon},
       title={Finitary codings for the random-cluster model and other
  infinite-range monotone models},
        date={2018},
     journal={arXiv preprint arXiv:1808.02333},
}

\bib{harvey2006universal}{article}{
      author={Harvey, Nate},
      author={Holroyd, Alexander~E},
      author={Peres, Yuval},
      author={Romik, Dan},
       title={Universal finitary codes with exponential tails},
        date={2006},
     journal={Proceedings of the London Mathematical Society},
      volume={94},
      number={2},
       pages={475\ndash 496},
}

\bib{holroyd2017one}{inproceedings}{
      author={Holroyd, Alexander~E},
       title={One-dependent coloring by finitary factors},
organization={Institut Henri Poincar{\'e}},
        date={2017},
   booktitle={Annales de l'institut henri poincar{\'e}, probabilit{\'e}s et
  statistiques},
      volume={53},
       pages={753\ndash 765},
}

\bib{holroyd2017mallows}{article}{
      author={Holroyd, Alexander~E},
      author={Hutchcroft, Tom},
      author={Levy, Avi},
       title={Mallows permutations and finite dependence},
        date={2017},
     journal={To appear in Annals of Probability, arXiv:1706.09526},
}

\bib{holroyd2018finitely}{article}{
      author={Holroyd, Alexander~E},
      author={Hutchcroft, Tom},
      author={Levy, Avi},
       title={Finitely dependent cycle coloring},
        date={2018},
     journal={Electronic Communications in Probability},
      volume={23},
}

\bib{holroyd2015symmetric}{article}{
      author={Holroyd, Alexander~E},
      author={Liggett, Thomas~M},
       title={Symmetric 1-dependent colorings of the integers},
        date={2015},
     journal={Electronic Communications in Probability},
      volume={20},
}

\bib{holroyd2016finitely}{inproceedings}{
      author={Holroyd, Alexander~E},
      author={Liggett, Thomas~M},
       title={Finitely dependent coloring},
organization={Cambridge University Press},
        date={2016},
   booktitle={Forum of mathematics, pi},
      volume={4},
}

\bib{holroyd2017finitary}{article}{
      author={Holroyd, Alexander~E},
      author={Schramm, Oded},
      author={Wilson, David~B},
       title={Finitary coloring},
        date={2017},
     journal={The Annals of Probability},
      volume={45},
      number={5},
       pages={2867\ndash 2898},
}

\bib{ibragimov1965independent}{article}{
      author={Ibragimov, IA},
      author={Linnik, YV},
       title={Independent and stationarily connected variables},
        date={1965},
     journal={Izdat.” Nauka”, Moscow},
}

\bib{ibragimov1975independent}{article}{
      author={Ibragimov, IA},
      author={Linnik, YV},
       title={Independent and stationary sequences of random variables},
        date={1971},
     journal={Wolters, Noordhoff Pub.},
}

\bib{keane1977class}{article}{
      author={Keane, Michael},
      author={Smorodinsky, Meir},
       title={A class of finitary codes},
        date={1977},
     journal={Israel Journal of Mathematics},
      volume={26},
      number={3-4},
       pages={352\ndash 371},
}

\bib{keane1979bernoulli}{article}{
      author={Keane, Michael},
      author={Smorodinsky, Meir},
       title={Bernoulli schemes of the same entropy are finitarily isomorphic},
        date={1979},
     journal={Annals of Mathematics},
      volume={109},
      number={2},
       pages={397\ndash 406},
}

\bib{knuth1976complexity}{article}{
      author={Knuth, Donald},
      author={Yao, Andrew},
       title={The complexity of nonuniform random number generation},
        date={1976},
     journal={Algorithm and Complexity, New Directions and Results},
       pages={357\ndash 428},
}

\bib{linial1987distributive}{inproceedings}{
      author={Linial, Nathan},
       title={Distributive graph algorithms global solutions from local data},
organization={IEEE},
        date={1987},
   booktitle={Foundations of computer science, 1987., 28th annual symposium
  on},
       pages={331\ndash 335},
}

\bib{lyons2017factors}{article}{
      author={Lyons, Russell},
       title={Factors of {IID} on trees},
        date={2017},
     journal={Combinatorics, Probability and Computing},
      volume={26},
      number={2},
       pages={285\ndash 300},
}

\bib{lyons2017probability}{book}{
      author={Lyons, Russell},
      author={Peres, Yuval},
       title={Probability on trees and networks},
   publisher={Cambridge University Press},
        date={2017},
      volume={42},
}

\bib{naor1991lower}{article}{
      author={Naor, Moni},
       title={A lower bound on probabilistic algorithms for distributive ring
  coloring},
        date={1991},
     journal={SIAM Journal on Discrete Mathematics},
      volume={4},
      number={3},
       pages={409\ndash 412},
}

\bib{smorodinsky1992finitary}{article}{
      author={Smorodinsky, Meir},
       title={Finitary isomorphism of m-dependent processes},
        date={1992},
     journal={Symbolic dynamics and its applications},
       pages={373\ndash 376},
}

\bib{spinka2018finitaryising}{article}{
      author={Spinka, Yinon},
       title={Finitary coding for the sub-critical {I}sing model with finite
  expected coding volume},
        date={2018},
     journal={To appear in Electronic Journal of Probability,
  arXiv:1801.02529},
}

\end{biblist}
\end{bibdiv}

\end{document}